\documentclass{elsarticle}

\usepackage{amsfonts}
\usepackage{amsthm}
\usepackage{amsmath}
\usepackage{amssymb}
\usepackage{mathtools} %for prescript
\usepackage{xcolor}
\usepackage{soul} %for strikethrough
\usepackage{mathrsfs}

\usepackage[top=1in, bottom=1in, left=1in, right=1in]{geometry}

\newcommand{\norm}[1]{\left \lVert#1\right \rVert}
\newtheorem{theorem}{Theorem}[section]
\newtheorem{lemma}{Lemma}[section]

\theoremstyle{definition}
\newtheorem{definition}{Definition}[section]

\theoremstyle{remark}
\newtheorem{remark}{Remark}[section]

\numberwithin{equation}{section}

\begin{document}
	
		\begin{frontmatter}
	
	\title{Perturbation theory for fractional evolution equations in a Banach space}
	
	%%this line removes the date, but space is still left for it;
	%if used, remove the \vspace{-1cm}
	\date{}
	
	%this gives the date in the form Mon 30 Jan 2012, 8:57pm;
	%if used, retain the \vspace{-1cm}
	%\date{\shortdayofweekname{\day}{\month}{\year}{ }\mydate\today}

	%\author[]{}
	\author[]{Arzu Ahmadova}
	\ead{arzu.ahmadova@emu.edu.tr}
	\author[]{Ismail T. Huseynov}
	\ead{ismail.huseynov@emu.edu.tr}
	\author[]{Nazim I. Mahmudov}
	\ead{nazim.mahmudov@emu.edu.tr}
	\cortext[cor1]{Corresponding author}
	
	\address{Department of Mathematics, Eastern Mediterranean University, Mersin 10, 99628, T.R. North Cyprus, Turkey}
	
	% Latex won't make the title unless told:
	%\maketitle
	
	%%to remove the space left for date, use:

	\begin{abstract}
		\noindent A strong inspiration for studying perturbation theory for fractional evolution equations comes from the fact that they have proven to be useful tools in modeling many physical processes. In this paper, we study fractional evolution equations of order $\alpha\in (1,2]$ associated with the infinitesimal generator of an operator fractional cosine function generated by bounded time-dependent perturbations in a Banach space. We show that the abstract fractional Cauchy problem associated with the infinitesimal generator $A$ of a strongly continuous fractional cosine function remains uniformly well-posed under bounded time-dependent perturbation of $A$. We also provide some necessary special cases.
	\end{abstract}
	
		\end{frontmatter}
	\section{Introduction}\label{Sec:intro}
	%In many cases, the evolution equation (or the associated linear operator) is given as a sum of several terms with different physical meanings and different mathematical properties.
	
	Some partial differential equations arising in the transverse motion of an extensible beam \cite{Fit}, the vibration of hinged bars \cite{W-K}, damped McKean-Vlasov equations \cite{MM} and many other physical phenomena can be formulated as the second-order abstract differential equations with cosine families in infinite-dimensional spaces.
	The most fundamental and extensive work on cosine families of linear operators is that of Fattorini in \cite{Fat-1,Fat-2}.

	Moreover, Travis and Webb \cite{TW1} have investigated the following semilinear second-order Cauchy problem in a Banach space $\mathbb{X}$ via the theory of strongly continuous cosine families of linear bounded operators:
	
	\begin{equation}\label{abstract-1}
		\begin{cases}
			u^{\prime \prime}(t)= Au(t) + f(t,u(t),u^{\prime}(t)), \quad t \in \mathbb{R},\\
			u(t_{0})=x \in \mathbb{X}, \quad u^{\prime}(t_{0})=y \in \mathbb{X}.
		\end{cases}
	\end{equation} 
	A systematic and general treatment of the abstract semilinear second-order Cauchy problem from the stand point of existence, uniqueness, continuous dependence and smoothness of solutions have been provided by Travis and Webb in \cite{TW1,TW2}. Also, Bochenek in \cite{bochenek} has investigated the existence of a solution of the initial value problem for the second-order abstract evolution equation under more general hypotheses than in \cite{TW1}.
	The fractional analogue of the same problem \eqref{abstract-1} with Caputo time-derivative is considered by Kexue \cite{kexue1} and general results are attained by using the theory of fractional cosine families.
	
	The existence and uniqueness of the following inhomogeneous $\alpha$-order Cauchy problem is studied by Kexue et al. \cite{kexue2} in a Banach space $\mathbb{X}$:
	
		\begin{equation*}
		\begin{cases*}
				\left( \prescript{}{}{\mathcal{D}^{\alpha}_{t}}u\right) (t)=A u(t)+f(t), \quad t >0, \\
				\left( \mathcal{I}^{2-\alpha}u\right) (0)=0, \quad  \left( \mathcal{I}^{1-\alpha}u\right) (0)=x \in \mathbb{X},
		\end{cases*}
	\end{equation*}
	 where  $1<\alpha<2$, $A: \mathscr{D}(A)\subseteq \mathbb{X} \to \mathbb{X}$ is a closed, densely-defined linear operator, $\mathscr{D}(A)$ is the domain of $A$ endowed with the graph norm $\norm{x}_{\mathscr{D}(A)}=\norm{x}+\norm{Ax}$, and $\prescript{}{}{\mathcal{D}^{\alpha}_{t}}$ is the $\alpha$-order Riemann-Liouville fractional derivative operator. Furthermore, in \cite{chen} Chen and Li have established the properties of three kinds of resolvent families defined by purely algebraic equations which extend the classical cosine functional equation. In \cite{henriquez}, Henr\'{i}quez et al. have studied the differentiability of mild solutions for a class of fractional abstract Cauchy problems of order $\alpha\in (1,2)$.

	The perturbation theory has long been a very useful tool in the hands of both the analyst and physicist. A considerable amount of research has been done on the perturbation theory of linear operators in a Banach space, principally by Phillips \cite{philips}, Lutz \cite{lutz}, Travis and Webb \cite{TW3}.
	
	In \cite{philips}, Philips has obtained the closed-form of the classical solution to the following perturbed Cauchy problem in a Banach space $\mathbb{X}$:
		\begin{equation*}
		\begin{cases*}
			u^{\prime}(t)= (A+B(t))u(t) + f(t), \quad t >0,\\
			u(0)=x \in \mathscr{D}(A),
		\end{cases*}
	\end{equation*}  
where $B:\mathbb{R_{+}}\to \mathcal{L}(\mathbb{X})$, $f:\mathbb{R_{+}}\to \mathbb{X}$ be strongly continuously differentiable functions on $\mathbb{R_{+}}$.

In \cite{lutz}, Lutz has investigated for the first time the effect on the homogeneous Cauchy problem associated with the infinitesimal generator of an operator cosine function produced by bounded time-dependent perturbations: 

\begin{equation}\label{secondorder}
	\begin{cases}
		u^{\prime \prime}(t)= (A+B(t))u(t), \quad t \in \mathbb{R},\\
		u(0)=x \in \mathscr{D}(A), \quad u^{\prime}(0)=y \in \mathscr{D}(A).
	\end{cases}
\end{equation}  
 Lin \cite{lin} has studied the time-dependent perturbation theory for abstract second-order evolution equations and their applications to partial differential equation in the form of the following problem:

\begin{equation}\label{abstract2}
	\begin{cases}
		u^{\prime \prime}(t)= \left( A+B(t)\right) u(t) + f(t), \quad t \in (0,\tau],\quad \tau >0,\\
		u(0)=u_{0},\quad u^{\prime}(0)=v_{0}.
	\end{cases}
\end{equation}
On the other hand, Serizawa and Watanabe \cite{watanabe2} have investigated the same problem stated in \eqref{abstract2} with Neumann boundary condition. 

The fractional analogue of the abstract problem \eqref{secondorder} has been established by Bazhlekova in \cite{bazhlekova,bazhlekova2}. Bazhlekova \cite{bazhlekova2} has proposed uniquely determined a classical solution of the following homogeneous Cauchy problem:
\begin{equation*}
	\begin{cases*}
		\left( \prescript{C}{}{\mathcal{D}^{\alpha}_{t}}u\right) (t)= (A+B(t))u(t), \quad t > 0,\\
		u(0)=x \in \mathscr{D}(A), \quad u^{\prime}(0)=0,
	\end{cases*}
\end{equation*} 
where $B:\mathbb{R_{+}}\to \mathcal{L}(\mathbb{X})$ is a continuous function on $\mathbb{R_{+}}$.
	
 % Moreover, the perturbation of strongly continuous cosine families has been studied in \cite{watanabe,webb,bobrowski}.
	
	%There are a few research articles dealing with fractional-order abstract Cauchy problem. In 2001, Bazhlekova's dissertation \cite{bazhlekova} investigated perturbation theory for fractional evolution equations in abstract spaces based on the constructed theory of Lutz's research \cite{lutz}. Thus, she considered only the homogeneous fractional abstract Cauchy problem of order $\alpha \in(1,2)$ in Banach space with the second initial condition equal to zero. To the best of our knowledge, we extend her work by considering the following fractional abstract Cauchy problem \eqref{eq-2}. 

    Motivated by above articles, we consider an abstract Cauchy problem of fractional order $\alpha\in (1,2]$ in a Banach space $\mathbb{X}$:
	
	\begin{equation}\label{eq-2}
	\left( \prescript{C}{}{\mathcal{D}^{\alpha}_{t}}u\right) (t)=\left(A+B(t) \right) u(t)+f(t), \quad t>0,
	\end{equation}
	with initial conditions
	
	\begin{align}\label{initial}
	&u(0)=x\in \mathscr{D}(A), \qquad  u^{\prime}(0)=y\in \mathscr{D}(A),
	\end{align}
which is uniformly well-posed. Here $\prescript{C}{}{\mathcal{D}^{\alpha}_{t}}$ is a fractional-order Caputo differentiation operator and \\ $A:\mathscr{D}(A)\subseteq \mathbb{X} \to \mathbb{X}$ be a densely-defined, closed linear operator in a Banach space $\mathbb{X}$. Let $\mathscr{D}(A)\subseteq \mathbb{X}$ denote the
	domain of $A$ which will be specified later, $B:\mathbb{R}_{+}\to \mathcal{L}(\mathbb{X})$ and $f:\mathbb{R}_{+}\to \mathbb{X}$ be strongly continuously differentiable functions on $\mathbb{R_{+}}$.
	
	Therefore, the aim of this paper is to develop the perturbation theory to study fractional-order abstract Cauchy problems of order $\alpha\in (1,2]$ which generalizes classical case \eqref{secondorder}. The pioneering work on \eqref{eq-2}-\eqref{initial} in classical sense was done by Phillips \cite{philips}, Lutz \cite{lutz} and our development follows these approaches. To the best of our knowledge, the presented results extend those of \cite{lutz,bazhlekova2} in several aspects. First, we allow for more general case, in that we assume $u^{\prime}(0)=y\in \mathscr{D}(A)$, i.e., the second initial value is non-zero and as a result of this choice we derive the corresponding perturbed fractional sine family $S_{\alpha}(t;A+B)$, however this condition is zero in \cite{bazhlekova2}. Second, we consider both homogeneous and inhomogeneous abstract Cauchy problems for fractional evolution equations, however Bazhlekova \cite{bazhlekova2} and Lutz \cite{lutz} have studied only homogeneous cases in fractional and classical senses, respectively. Furthermore, we have attained new closed-form of solutions whenever $A,B \in \mathcal{L}(\mathbb{X})$ are non-permutable linear bounded operators. Thus, in this work it is shown that the inhomogeneous Cauchy problem with the infinitesimal generator $A$ of a strongly continuous fractional cosine families remains uniformly well-posed under bounded time-dependent perturbations.
	%and if $x \in \mathscr{D}(A), f G C ([0,t]; X), then the equation

	%Ball \cite{ball} defined weak solutions of the equation \begin{equation}\label{abstract}
	%	\begin{cases}
	%		u^{\prime}(t)= Au(t) + f(t), \quad t \in (0,\tau],\\
		%	u(0)=x,
		%\end{cases}
	%\end{equation}
	%and showed that a necessary and sufficient condition for the existence of unique weak solutions for any initial data in the Banach space $\mathbb{X}$ is that $A$ generates a strongly continuous semigroup on $\mathbb{X}$, and that in this case the solution is given by the variation of the constant formula.
	
	In this paper, we study perturbation properties of problem \eqref{eq-2} with initial conditions \eqref{initial}, generalizing some facts concerning fractional cosine operator functions
	and also distinguishing some new features. Therefore, the structure of our paper is organized as follows. In Section \ref{prel}, we provide main definitions and relations from fractional calculus, fractional evolution equations and operator theory. Section \ref{pertub} is devoted to studying homogeneous and inhomogeneous abstract fractional-order perturbed Cauchy problem in a Banach space $\mathbb{X}$. In Section \ref{sec4}, we investigate some necessary particular cases of the main problem \eqref{eq-2}-\eqref{initial} and we give some comparisons between existing and our results.

\section{Preliminaries}\label{prel}
We embark on this section by briefly introducing the essential structure of fractional calculus, fractional evolution equations and operator theory for linear operators. For the more salient details on these matters, see the textbooks \cite{kilbas,podlubny,engel,kato,watanabe,voigt,miyadera}.

Let $\mathbb{R_{+}}=[0,\infty)$ and $\mathbb{N}$ denote the set of natural numbers with $\mathbb{N}_{0}=\mathbb{N}\cup\left\lbrace 0\right\rbrace$.
Let $\mathbb{\mathbb{X}}$ be Banach space equipped with the norm $\norm{\cdot}$. 
We donote by $\mathcal{L}(\mathbb{X})$ the Banach algebra of all bounded linear operators on $\mathbb{X}$ and becomes a Banach space with regard to the norm $\norm{T}=\sup\left\lbrace \norm{Tx}: \norm{x} \leq 1\right\rbrace $, for any $T\in \mathcal{L}(\mathbb{X})$. The identity and zero operators on $\mathbb{X}$ are denoted by $I\in \mathcal{L}(\mathbb{X})$ and $0\in \mathcal{L}(\mathbb{X})$, respectively.

We will use the following functional spaces \cite{hille-philips} through the paper:

\begin{itemize}
	%\item $\mathbb{R}^{n}$ denotes a Euclidean space endowed with a $1$-norm $\norm{x}_{1}=\sum\limits_{i=1}^{n}|x_{i}|$ for any $x=\left(x_{1},\ldots,x_{n}\right)\in\mathbb{R}^{n }$;
	
	\item
	$\mathbb{C}\left(\mathbb{R_{+}},\mathbb{X}\right)$ denotes the Banach space of continuous $\mathbb{X}$-valued functions $g:\mathbb{R_{+}}\to \mathbb{X} $ equipped with an infinity norm $\|g\|_{\infty}=\sup\limits_{t\in \mathbb{R_{+}}}\norm{g(t)}$;
	
	\item
	$\mathbb{L}^{1}(\mathbb{R_{+}},\mathbb{X})$ denotes a Banach space of equivalence classes of all measurable $\mathbb{X}$-valued functions \\ $g:\mathbb{R_{+}}\to \mathbb{X}$ which are Bochner integrable and normed by $\norm{g}_{\mathbb{L}_{1}}=\int\limits_{0}^{T}\norm{g(s)}\mathrm{d}s<\infty$; 
	
	\item
	$\mathbb{C}^{n}(\mathbb{R_{+}},\mathbb{X})$, $n\in \mathbb{N}$ denotes the Banach space of $n$-times continuously differentiable $\mathbb{X}$-valued functions defined by 
	\begin{equation*}
	\mathbb{C}^{n}(\mathbb{R_{+}},\mathbb{X})=\left\lbrace g\in\mathbb{C}^{n}(\mathbb{R_{+}},\mathbb{X}): g^{(n)} \in\mathbb{C}(\mathbb{R_{+}},\mathbb{X}), n \in \mathbb{N}\right\rbrace
	\end{equation*}
	and equipped with an infinity norm $\|g\|_{\infty}=\sum\limits_{k=0}^{n}\sup\limits_{t\in \mathbb{R_{+}}}\norm{g^{(k)}(t)}$.
	In addition, $\mathbb{C}^{n}(\mathbb{R_{+}},\mathbb{X})\subset \mathbb{C}(\mathbb{R_{+}},\mathbb{X})$, $n\in \mathbb{N}$.
\end{itemize}

Let $f:\mathbb{R_{+}}\to\mathbb{R}$ be an integrable scalar-valued function and let $g:\mathbb{R_{+}}\to \mathbb{X}$ be a continuous $\mathbb{X}$-valued function. We denote by 
\begin{equation*}
(f\ast g)(t)=\int\limits_{0}^{t}f(t-s)g(s)\mathrm{d}s, \quad t \in \mathbb{R_{+}},
\end{equation*}
the convolution operator of $f$ and $g$. Furthermore, if $T:\mathbb{R_{+}}\to\mathscr{L}(\mathbb{X})$ is a strongly continuous operator-valued map, we define
\begin{equation*}
(T\ast g)(t)=\int\limits_{0}^{t}T(t-s)g(s)\mathrm{d}s=\int\limits_{0}^{t}T(s)g(t-s)\mathrm{d}s, \quad t \in \mathbb{R_{+}}.
\end{equation*}

\begin{lemma}\cite{engel}
Let $\mathbb{J}$ be some real interval and $P,Q: \mathbb{J}\to \mathcal{L}(\mathbb{X})$ be two strongly continuous operator-valued functions defined on $\mathbb{J}$. Furthermore, suppose that $P(\cdot)x:\mathbb{J}\to \mathbb{X}$ and $Q(\cdot)x:\mathbb{J}\to \mathbb{X}$ are differentiable for all $x\in \mathscr{D}$ for some subspace $\mathscr{D}$ of $\mathbb{X}$, which is invariant under $Q$. Then $(PQ)(\cdot)x:\mathbb{J}\to \mathbb{X}$, defined by $(PQ)(t)x\coloneqq P(t)Q(t)x$, is differentiable for every $x\in \mathscr{D}$ and 
\begin{equation}\label{lem3.1}
	\frac{\mathrm{d}}{\mathrm{d}t}\Big(P(\cdot)Q(\cdot)x\Big)(t_{0})=\frac{\mathrm{d}}{\mathrm{d}t}\Big(P(\cdot)Q(t_{0})x\Big)(t_{0})+P(t_{0})\Big(\frac{\mathrm{d}}{\mathrm{d}t}Q(\cdot)x\Big)(t_{0}).
\end{equation} 
\end{lemma}

 \begin{definition}\cite{kilbas,podlubny}\label{RLI}
	The Riemann-Liouville fractional integral of  order $\alpha > 0$ for a function \\ $f\in  \mathbb{L}^{1}\left( \mathbb{R_{+}}, \mathbb{X}\right) $  is defined by
	\begin{equation}
	\prescript{}{}(\mathcal{I}^{\alpha}_{t}f)(t)=\left( g_{\alpha}\ast f\right) (t)=\frac{1}{\Gamma(\alpha)}\int\limits_0^t(t-s)^{\alpha-1}f(s)\,\mathrm{d}s \\, \quad t>0.
	\end{equation}
By setting $\left( \mathcal{I}^{0}_{t}f\right) (t)=f(t)$.	For the sake of brevity, we use the following notation for $\alpha\geq 0$:
	
	\begin{align*}
	g_{\alpha}(t)=
	\begin{cases*}
	\frac{t^{\alpha-1}}{\Gamma(\alpha)}, \quad t>0,\\
	0, \qquad t \leq 0,
	\end{cases*}
	\end{align*}
where $\Gamma:\mathbb{R_{+}}\to \mathbb{R}$ is the well-known Euler's gamma function. Note that $g_{0}(t)=0$, since $\frac{1}{\Gamma(0)}=0$, $t\in \mathbb{R}$. These
	functions satisfy the semigroup property:
	\begin{equation}\label{galpha}
	(g_{\alpha}\ast g_{\beta})(t)=g_{\alpha+\beta}(t), \quad t \in \mathbb{R}.
	\end{equation}
	
	Moreover, the Riemann-Liouville fractional integral operators $\left\lbrace \mathcal{I}^{\alpha}_{t} \right\rbrace_{\alpha \geq 0}$ satisfy the semigroup property
	\begin{equation}
	\mathcal{I}^{\alpha}_{t}\mathcal{I}^{\beta}_{t}=\mathcal{I}^{\alpha+\beta}_{t}, \quad \alpha,\beta \geq 0,\quad t>0.
	\end{equation} 
\end{definition}
\begin{definition}\cite{kilbas,podlubny}\label{RLD}
	The Riemann-Liouville fractional derivative of order $n-1<\alpha\leq n$ for a function \\ $\ f\in  \mathbb{C}^{n}\left( \mathbb{R_{+}}, \mathbb{X}\right) $ is  defined by 
	\begin{equation}
	\left( \prescript{}{}{\mathcal{D}}^{\alpha}_{t}f\right)(t)=\prescript{}{}{\mathcal{D}^{n}_{t}}\left( g_{n-\alpha}\ast f\right) (t)=\prescript{}{}{\mathcal{D}^{n}_{t}}\left( \mathcal{I}^{n-\alpha}_{t}f\right) (t),\quad t>0,
	\end{equation}
	where $\prescript{}{}{\mathcal{D}^{n}_{t}}=\frac{\mathrm{d}^{n}}{\mathrm{d}t^{n}}$ for $n\in \mathbb{N}$.

	We sometimes adopt the convention
	\begin{equation*} 
	\left( \prescript{}{}{\mathcal{D}}^{-\alpha}_{t}f\right)(t)=\left( \mathcal{I}^{\alpha}_{t}f\right)(t),\quad\alpha>0, \quad t>0,
	\end{equation*}
	under which the Riemann-Liouville fractional derivative is the analytic continuation of the Riemann-Liouville fractional integral for positive $\alpha$.
\end{definition}

\begin{definition}\cite{kilbas,podlubny}\label{CD}
	The Caputo fractional derivative  of a function $f\in \mathbb{C}^{n}\left( \mathbb{R_{+}},\mathbb{X}\right) $ with fractional order \\ $n-1<\alpha\leq n$, $n\in\mathbb{N}$ is defined by
	\begin{equation*}
	\left( \prescript{C}{}{\mathcal{D}}^{\alpha}_{t}f\right)(t)=\Big(g_{n-\alpha}\ast \mathcal{D}^{n}_{t}f\Big)(t)= \mathcal{I}^{n-\alpha}_{t}\prescript{}{}{\mathcal{D}^{n}_{t}}f(t),\quad t>0.
	\end{equation*}
 In particular, for $\alpha\in(1,2]$, the definition for $f\in \mathbb{C}^{2}\left( \mathbb{R_{+}},\mathbb{X}\right) $ is given by:
	\begin{equation*}
	\left( \prescript{C}{}{\mathcal{D}}^{\alpha}_{t}f\right)(t)=\left( g_{2-\alpha}\ast \mathcal{D}^{2}_{t}f\right) (t)=\frac{1}{\Gamma(2-\alpha)}\int\limits_0^t(t-s)^{1-\alpha}f^{\prime \prime}(s)\,\mathrm{d}s, \quad t>0.
	\end{equation*}
\end{definition}

\begin{lemma}\cite{kilbas,podlubny}\label{RLC}
The relationship between Riemann-Liouville and Caputo fractional differential operators  of order $ n-1<\alpha\leq n$, $n\in\mathbb{N}$ for $f\in \mathbb{C}^{n}\left( \mathbb{R_{+}},\mathbb{X}\right) $ is given by:
\begin{equation}\label{RLandC} 
\left( \prescript{C}{}{\mathcal{D}}^{\alpha}_{t}f\right)  (t)= \prescript{}{}{\mathcal{D}}^{\alpha}_{t}\left( f(t)-\sum_{i=0}^{n-1}\frac{f^{(i)}(0)t^{i}}{\Gamma(i+1)}\right) .
\end{equation}
\end{lemma}
By making use of the above property, we attain the relation of order $ 1<\alpha\leq 2$ for $f\in \mathbb{C}^{2}\left( \mathbb{R_{+}},\mathbb{X}\right) $:
\begin{equation} \label{gcon}
\prescript{}{}{\mathcal{I}}^{\alpha}_{t}\left( \prescript{C}{}{\mathcal{D}}^{\alpha}_{t}f\right)  (t)=\left( g_{\alpha}\ast g_{2-\alpha} \ast \mathcal{D}^{2}_{t}f\right) (t)= f(t)-f(0)-tf^{\prime}(0).
\end{equation}

\begin{remark}
	Since $f$ is an abstract function with values in $\mathbb{X}$, the integrals which appear in Definition \ref{RLI},  \ref{RLD}, \ref{CD} and Lemma \ref{RLC} are taken in Bochner's sense.
\end{remark}

\begin{definition}\cite{engel} \label{laplace}
If $f:\mathbb{R_{+}}\to \mathbb{X}$ is measurable and exponentially bounded function of exponent $\omega \in \mathbb{R}$, i.e., $\|f(t)\|\leq Me^{\omega t}$ for all $t\geq 0$ and some constant $M>0$. Then the Laplace transform \\ $\mathscr{L} f:\left\lbrace \lambda \in \mathscr{C} : Re(\lambda)>\omega \right\rbrace\to \mathbb{X}$ is defined by
\end{definition}

\begin{equation*}
	\left( \mathscr{L} f\right)  (\lambda)=\hat{f}(\lambda)=\int\limits_{0}^{\infty}e^{-\lambda t}f(t)\mathrm{d}t,
\end{equation*}
where $Re(\lambda)$ represents the real part of the complex number $\lambda$.
The Laplace transform of the Caputo's fractional differentiation operator \cite{kilbas,podlubny} is defined by
\begin{equation} \label{lapCaputo}
\left( \mathscr{L} (\prescript{C}{}{\mathcal{D}}^{\alpha}_{t}f)\right) (\lambda) =\lambda^{\alpha}\hat{f}(\lambda)-\sum_{k=1}^{n}\lambda^{\alpha-k}f^{(k-1)}(0), \quad n-1<\alpha\leq n, \quad n\in\mathbb{N}.
\end{equation}

In the particular case, the Laplace integral transform of the Caputo fractional derivative for  $\alpha \in (1,2]$  is:
\begin{equation*}
	\left( \mathscr{L} (\prescript{C}{}{\mathcal{D}}^{\alpha}_{t}f)\right) (\lambda) =\lambda^{\alpha}\hat{f}(\lambda)-\lambda^{\alpha-1}f(0)-\lambda^{\alpha-2}f^{\prime}(0) .
\end{equation*}

The Mittag-Leffler function is a natural generalization of the exponential function, first proposed as a single and two parameter functions of one variable by using an infinite series.

\begin{definition}\cite{Gorenflo}\label{Def:ML}
	The classical Mittag-Leffler function is defined by
	\begin{equation}\label{ML1}
	E_{\alpha}(t)= \sum_{k=0}^{\infty}\frac{t^{k}}{\Gamma(k \alpha +1)}, \quad  \alpha>0, \quad t\in\mathbb{R}.
	\end{equation}

	The two-parameter Mittag-Leffler function is given by
	\begin{equation}\label{ML-2}
	E_{\alpha,\beta}(t)= \sum_{k=0}^{\infty}\frac{t^{k}}{\Gamma(k \alpha +\beta)}, \quad  \alpha>0, \quad \beta\in\mathbb{R}, \quad t\in\mathbb{R}.
	\end{equation}
	 It is important to note that
	\[
	 E_{\alpha,1}(t)=E_{\alpha}(t),\quad E_1(t)=\exp(t), \quad t\in\mathbb{R}.
	\]
\end{definition}

$\star$ \textbf{Strongly continuous cosine and sine families of linear operators}

\begin{definition}\cite{TW1}
	\textbf{1)} A one parameter family $\left\lbrace C(t;A):t \in \mathbb{R}\right\rbrace \in \mathcal{L}(\mathbb{X}) $  into itself is called a strongly continuous cosine family in the Banach space $\mathbb{X}$ if and only if
\begin{itemize}
	\item $C(0;A)=I$;
	\item $C(s+t;A)+C(s-t;A)=2C(s;A)C(t;A)$ for all $s,t\in\mathbb{R}$;
	\item $C(t;A)x$ is continuous in $t$ on $\mathbb{R}$ for each fixed $x \in \mathbb{X}$.
\end{itemize}

\textbf{2)} The corresponding strongly continuous sine family $\left\lbrace S(t;A):t \in \mathbb{R}\right\rbrace \in \mathcal{L}(\mathbb{X}) $ is defined by
\begin{equation}
S(t;A)x=\int\limits_{0}^{t}C(s;A)x\mathrm{d}s, \quad x \in \mathbb{X}, \quad t \in \mathbb{R}.
\end{equation}

\textbf{3)} The linear operator $A: \mathscr{D}(A)\subseteq \mathbb{X} \to \mathbb{X}$ defined by
\begin{equation}
Ax=2\lim\limits_{t \to 0_{+}}\frac{C(t;A)x-x}{t^{2}}, \quad x \in \mathscr{D}(A)=\left\lbrace x \in \mathbb{X}: C(\cdot;A)x \in \mathbb{C}^{2}(\mathbb{R},\mathbb{X})\right\rbrace ,
\end{equation}

is called  the infinitesimal generator of a strongly continuous cosine family $\left\lbrace C(t;A):t \in \mathbb{R}\right\rbrace$ and $\mathscr{D}(A)$ is a domain of $A$.
\end{definition}

It is known that the infinitesimal generator $A$ is closed, densely-defined operator on $\mathbb{X}$. We will also make use of the set
\[
\mathcal{E}=\left\lbrace  x \in \mathbb{X}: C(\cdot;A)x \in \mathbb{C}^{1}\left( \mathbb{R},\mathbb{X}\right)\right\rbrace .
\] 
The families $\left\lbrace  C(t;A):t\in \mathbb{R}\right\rbrace  $ and $\left\lbrace  S(t;A):t\in \mathbb{R}\right\rbrace  $ have the following properties:

\begin{theorem}\cite{TW1}
	Let $\{ C(t;A):t\in \mathbb{R}\} $ be a strongly continuous cosine family in $\mathbb{X}$ with infinitesimal generator $A$. The following assertions hold true:
\begin{itemize} 
	\item There exists constants $M\geq 1$ and $\omega \geq 0$ such that $\norm{C(t;A)}\leq Me^{\omega |t|}$ for all $t \in \mathbb{R}$;
	\item  If $x \in \mathbb{X}$, then $S(t;A)x \in \mathcal{E}$;
	\item  If $x \in \mathcal{E}$, then $S(t;A)x\in \mathscr{D}(A)$ and $\mathcal{D}^{1}_{t}C(t;A)x=AS(t;A)x$;
	\item  If $x\in \mathscr{D}(A)$, then $C(t;A)x\in \mathscr{D}(A)$ and $\mathcal{D}^{2}_{t}C(t;A)x=AC(t;A)x=C(t;A)Ax$;
	\item  If $x\in \mathcal{E}$, then $\lim\limits_{t \to 0} AS(t;A)x=0$;
	\item If $x\in \mathcal{E}$, then $S(t;A)x \in \mathscr{D}(A)$ and $\mathcal{D}^{2}_{t}S(t;A)x=AS(t;A)x$;
	\item If $x \in \mathscr{D}(A)$, then $S(t;A)x \in \mathscr{D}(A)$ and $AS(t;A)x=S(t;A)Ax$.
\end{itemize}
\end{theorem}

\begin{theorem}\cite{TW1}
	Let $\{ C(t;A):t\in \mathbb{R}\} $ be a strongly continuous cosine family in $\mathbb{X}$ satisfying $\norm{C(t;A)}\leq Me^{\omega |t|}$ for all $t\in\mathbb{R}$ and let $A$ be the infinitesimal generator of $\{ C(t;A):t\in \mathbb{R}\}$. Then for $Re(\lambda)>\omega$, $\lambda^{2}$ is in the resolvent set of $A$ and 
	\begin{align}
	&\lambda \mathcal{R}(\lambda^
	{2};A)x=\int\limits_{0}^{\infty}e^{-\lambda t}C(t;A
	)x\mathrm{d}t, \quad x \in \mathbb{X},\\
	& \mathcal{R}(\lambda^{2};A)x=\int\limits_{0}^{\infty}e^{-\lambda t}S(t;A)x\mathrm{d}t, \quad x \in \mathbb{X}.
	\end{align} 
\end{theorem}

\begin{theorem}\cite{TW1}
Let $\{ C(t;A):t\in \mathbb{R}\} $ be a strongly continuous cosine family in $\mathbb{X}$ with infinitesimal generator $A$. If $f: \mathbb{R} \to \mathbb{X}$ is continuously differentiable, $x_{0} \in \mathscr{D}(A)$, $x^{\prime}_{0}\in \mathcal{E}$, and
\begin{equation}
x(t)\coloneqq C(t;A)x_0+S(t;A)x^{\prime}_{0}+\int\limits_{0}^{t}S(t-s;A)f(s)\mathrm{d}s, \quad  t \in \mathbb{R},
\end{equation} 
then $x(t)\in \mathscr{D}(A)$ for $t\in\mathbb{R}$, $x$ is twice continuously differentiable, and $x$ satisfies
\begin{equation}
\begin{cases}
\mathcal{D}^{2}_{t}x(t)=Ax(t)+f(t),\quad t \in \mathbb{R},\\
x(0)=x_{0}, \quad x^{\prime}(0)=x^{\prime}_{0}.
\end{cases}
\end{equation}
\end{theorem}

$\star$ \textbf{Strongly continuous fractional cosine, sine and Riemann-Liouville families of linear operators}
\begin{definition}\cite{kexue1}
Let $1<\alpha < 2$. A family $C_{\alpha}(\cdot; A):\mathbb{R_{+}}\to \mathcal{L}(\mathbb{X})$ of all bounded linear operators on $\mathbb{X}$ is called a fractional cosine family if it satisfies the following assumptions:
\begin{itemize}
	\item $C_{\alpha}(t;A)$ is strongly continuous for all $t \in \mathbb{R_{+}}$ and $C_{\alpha}(0;A)=I$;
	\item $C_{\alpha}(s;A)C_{\alpha}(t;A)=C_{\alpha}(t;A)C_{\alpha}(s;A)$ for all $s,t\in \mathbb{R_{+}}$;
	\item The functional equation
	\begin{equation}
	C_{\alpha}(s;A)\mathcal{I}_{t}^{\alpha}C_{\alpha}(t;A)-\mathcal{I}_{s}^{\alpha}C_{\alpha}(s;A)C_{\alpha}(t;A)=\mathcal{I}_{t}^{\alpha}C_{\alpha}(t;A)-\mathcal{I}_{s}^{\alpha}C_{\alpha}(s;A) \quad \text{holds for all} \quad s,t \in \mathbb{R_{+}}.
	\end{equation}
\end{itemize}
\end{definition}
The linear operator $A$ is defined by
\begin{align}
Ax \coloneqq \Gamma(\alpha+1)\lim_{t \to 0_{+}}\frac{C_{\alpha}(t;A)x-x}{t^{\alpha}}, \quad x \in \mathscr{D}(A)\coloneqq \left\lbrace x \in \mathbb{X}: C_{\alpha}(\cdot;A) \in \mathbb{C}^{2}\left( \mathbb{R_{+}},\mathbb{X}\right) \right\rbrace,
\end{align}
where $A$ is the generator of the strongly continuous fractional cosine family $\{ C_{\alpha}(t;A):t\in \mathbb{R_{+}}\} $.

\begin{definition}\cite{kexue1}
	A strongly continuous fractional cosine family $\{ C_{\alpha}(t;A):t\in \mathbb{R_{+}}\} $ is said to be exponentially bounded if there exists constants $M\geq 1$, $\omega \geq 0$ such that
	\begin{equation}
	\norm{C_{\alpha}(t;A)}\leq M e^{\omega t}, \quad  t \in \mathbb{R_{+}}. 
	\end{equation} 
\end{definition}
\begin{theorem}\cite{chen}
Let $\{ C_{\alpha}(t;A):t\in \mathbb{R_{+}}\} $ be a  strongly continuous  fractional cosine family generated by the operator $A$. The following assertions hold true:
\begin{itemize}
	\item $C_{\alpha}(t;A)$ commutes with $A$ which means that $C_{\alpha}(t;A)\mathscr{D}(A) \subset \mathscr{D}(A) $ and $AC_{\alpha}(t;A)x=C_{\alpha}(t;A)Ax$ for all $x \in \mathscr{D}(A)$ and $t \in \mathbb{R_{+}}$;
	\item For all $x \in \mathbb{X}$, $\mathcal{I}_{t}^{\alpha}C_{\alpha}(t;A)x \in \mathscr{D}(A)$ and 
	\begin{equation}\label{2.19}
	C_{\alpha}(t;A)x=x+A\mathcal{I}_{t}^{\alpha}C_{\alpha}(t;A)x, \quad  t \in \mathbb{R_{+}}.
	\end{equation}
	\item For all $x \in \mathscr{D}(A)$,
	\begin{equation}
	C_{\alpha}(t;A)x=x+\mathcal{I}_{t}^{\alpha}C_{\alpha}(t;A)Ax, \quad  t \in \mathbb{R_{+}}.
	\end{equation}
	\item $A$ is closed, densely-defined operator on $\mathbb{X}$;
	%\item $C_{\alpha}(t;A)x$, $t \in \mathbb{R_{+}}$ is a unique mild (strong)
	%solution of the following abstract Cauchy problem for each $x \in \mathbb{X}$ ($x \in \mathscr{D}(A)$):
	%\begin{equation}
	%\begin{cases}
	%\prescript{C}{}{\mathcal{D}}_{t}^{\alpha}x(t)=Ax(t), \quad t>0,\\
	%x(0)=x, \quad x^{\prime}(0)=0.
	%\end{cases}
	%\end{equation}
\end{itemize}
\end{theorem}
\begin{definition}\cite{kexue1}
	The strongly continuous fractional sine family $S_{\alpha}(\cdot;A):\mathbb{R_{+}}\to \mathcal{L}(\mathbb{X})$ associated with $C_{\alpha}$ is defined by
	\begin{equation}
	S_{\alpha}(t;A)x=\int\limits_{0}^{t}C_{\alpha}(s;A)x\mathrm{d}s, \quad x \in \mathbb{X}, \quad t \in \mathbb{R_{+}}.
	\end{equation}
This implies that 
\begin{equation}\label{AA}
	\mathcal{D}^{1}_{t}S_{\alpha}(t;A)x=C_{\alpha}(t;A)x, \quad x \in \mathbb{X}, \quad t \in \mathbb{R_{+}}.
\end{equation}
\end{definition}

\begin{definition}\cite{kexue1}
	The strongly continuous fractional Riemann-Liouville family $T_{\alpha}(\cdot;A):\mathbb{R_{+}}\to \mathcal{L}(\mathbb{X})$ associated with $C_{\alpha}$ is defined by
	\begin{equation}\label{2.22}
	T_{\alpha}(t;A)x=\mathcal{I}_{t}^{\alpha-1}C_{\alpha}(t;A)x,\quad x \in \mathbb{X}, \quad t \in \mathbb{R_{+}}.
	\end{equation}
\end{definition}

For a strongly continuous fractional cosine family $\left\lbrace C_{\alpha}(t;A): t\in \mathbb{R_{+}}\right\rbrace$, we define 

\begin{equation*}
\mathcal{E}=\left\lbrace  x \in \mathbb{X}: C_{\alpha}(\cdot;A)x \in \mathbb{C}^{1}\left( \mathbb{R_{+}},\mathbb{X}\right)\right\rbrace.
\end{equation*} 

It is known that $T_{\alpha}(t;A)\mathcal{E}\subset \mathscr{D}(A)$ for all $t\in\mathbb{R_{+}}$ such that
\begin{equation}\label{2.25}
\mathcal{D}^{1}_{t}C_{\alpha}(t;A)x=AT_{\alpha}(t;A)x, \quad x \in \mathcal{E}, \quad t \in \mathbb{R_{+}}.
\end{equation}
From \eqref{2.25} (or \eqref{2.19} and \eqref{2.22}), for $x\in \mathbb{X}$, $\mathcal{I}_{t}^{1}T_{\alpha}(t;A)x\in \mathscr{D}(A)$ such that
\begin{equation}\label{2.26}
	A\mathcal{I}_{t}^{1}T_{\alpha}(t;A)x=C_{\alpha}(t;A)x-x.
\end{equation}
\begin{theorem}\cite{kexue1}
	Let $\{ C_{\alpha}(t;A):t\in \mathbb{R_{+}}\} $ be a strongly continuous fractional cosine family in $\mathbb{X}$ satisfying $\norm{C_{\alpha}(t;A)}\leq Me^{\omega t}$ for all $t\in\mathbb{R_{+}}$ and let $A$ be the infinitesimal generator of $\{ C_{\alpha}(t;A):t\in \mathbb{R_{+}}\}$. Then for $Re(\lambda)>\omega$, $\lambda^{\alpha}$ for $\alpha \in (1,2]$ is in the resolvent set of $A$ and 
	\begin{align}
	&\lambda^{\alpha-1} \mathcal{R}(\lambda^
	{\alpha};A)x=\int\limits_{0}^{\infty}e^{-\lambda t}C_{\alpha}(t;A
	)x\mathrm{d}t, \quad x \in \mathbb{X},\\
	& \lambda^{\alpha-2}\mathcal{R}(\lambda^{\alpha};A)x=\int\limits_{0}^{\infty}e^{-\lambda t}S_{\alpha}(t;A)x\mathrm{d}t, \quad x \in \mathbb{X},\\
	&\mathcal{R}(\lambda^{\alpha};A)x=\int\limits_{0}^{\infty}e^{-\lambda t}T_{\alpha}(t;A)x\mathrm{d}t, \quad x \in \mathbb{X}.
	\end{align} 
\end{theorem}

\begin{theorem}\cite{kexue1}
	Let $\{ C_{\alpha}(t;A):t\in \mathbb{R_{+}}\} $ be a strongly continuous fractional cosine family in $\mathbb{X}$ with the infinitesimal generator $A$. If $f: \mathbb{R_{+}} \to \mathbb{X}$ is continuously differentiable, $x_{0} \in \mathscr{D}(A)$, $x^{\prime}_{0}\in \mathcal{E}$, and
	\begin{equation}
	x(t)\coloneqq C_{\alpha}(t;A)x_0+S_{\alpha}(t;A)x^{\prime}_{0}+\int\limits_{0}^{t}T_{\alpha}(t-s;A)f(s)\mathrm{d}s, \quad  t \in \mathbb{R_{+}},
	\end{equation} 
	then $x(t)\in \mathscr{D}(A)$ for $t\in\mathbb{R}_{+}$, $x$ is twice continuously differentiable, and $x$ satisfies
	\begin{equation}
	\begin{cases}
	\prescript{C}{}{\mathcal{D}}^{\alpha}_{t}x(t)=Ax(t)+f(t),\quad t \in \mathbb{R_{+}},\\
	x(0)=x_{0}, \quad x^{\prime}(0)=x^{\prime}_{0}.
	\end{cases}
	\end{equation}
\end{theorem}

\section{Perturbation theory for fractional evolution equations}\label{pertub}

We first consider the following Cauchy problem for homogeneous fractional evolution equation in a Banach space $\mathbb{X}$:
\begin{equation}\label{eq-3}
\begin{cases}
\left( \prescript{C}{}{\mathcal{D}^{\alpha}_{t}}u\right) (t)=\left(A+B(t) \right) u(t), \quad t >0,\\
u(0)=x , \qquad  u^{\prime}(0)=y.
\end{cases}
\end{equation}

\begin{theorem}\label{thm1}
	Let $A$ be the infinitesimal generator of the fractional cosine family $\left\lbrace C_{\alpha}(t;A): t \in \mathbb{R_{+}}\right\rbrace $ with \\ $\|C_{\alpha}(t;A)\|\leq Me^{\omega t}$ for all $t \in \mathbb{R_{+}}$, and let $\left\lbrace S_{\alpha}(t;A): t \in \mathbb{R_{+}}\right\rbrace $ denote the fractional sine family associated with $C_{\alpha}$. Furthermore, suppose that $B:\mathbb{R_{+}}\to \mathcal{L}(\mathbb{X})$ is a strongly continuously differentiable function on $\mathbb{R_{+}}$. Then
	
	\begin{itemize}
	\item The Cauchy problem for the $\mathbb{X}$-valued differential equation \eqref{eq-3}
	is uniformly well-posed; more precisely, for each pair $x,y\in \mathscr{D}(A)$ there is a uniquely determined solution $u\in \mathbb{C}^{2}\left(\mathbb{R_{+}};\mathscr{D}(A) \right)$ of \eqref{eq-3} fulfilling initial conditions and $u\in\mathscr{D}(A)$ depends continuously (with respect to the topology of uniform convergence) on compact subsets of $\mathbb{R}_{+}$ upon $(x,y)$. 
	\item $u\in\mathscr{D}(A)$ is given by
	\begin{equation}\label{usol}
	u(t)= C_{\alpha}(t;A+B)x+S_{\alpha}(t;A+B)y, \quad t \in \mathbb{R_{+}},
	\end{equation}
	where for $x\in \mathscr{D}(A)$, $t \mapsto C_{\alpha}(t;A+B)x$ and $t\mapsto S_{\alpha}(t;A+B)x$ are $2$-times strongly continuously differentiable solutions of \eqref{eq-3} fulfilling 
	\begin{align}\label{csder}
	&C_{\alpha}(0;A+B)x=x, \quad C^{\prime}_{\alpha}(0;A+B)x=0,\\
	&S_{\alpha}(0;A+B)x=0,\label{ssder} \quad S^{\prime}_{\alpha}(0;A+B)x=x.
	\end{align}
	\item $C_{\alpha}(\cdot;A+B),S_{\alpha}(\cdot;A+B)\in \mathcal{L}(\mathbb{X})$  are defined by
	\begin{align}
	&C_{\alpha}(t;A+B)x\coloneqq \sum_{n=0}^{\infty}C_{\alpha,n}(t,A)x, \quad x \in \mathbb{X},\label{Cdefine}\\
	&S_{\alpha}(t;A+B)x\coloneqq \sum_{n=0}^{\infty}S_{\alpha,n}(t;A)x, \quad x\in \mathbb{X}\label{Sdefine},
	\end{align}
	where  
	\begin{align}
	&C_{\alpha,0}(t;A)x\coloneqq C_{\alpha}(t;A)x,\quad S_{\alpha,0}(t;A)x\coloneqq S_{\alpha}(t;A)x,\nonumber\\
	&C_{\alpha,n}(t;A)x\coloneqq \int\limits_{0}^{t}T_{\alpha}(t-s;A)B(s)C_{\alpha,n-1}(s;A)x\mathrm{d}s,\label{Calphan}\\
    &S_{\alpha,n}(t;A)x\coloneqq \int\limits_{0}^{t}T_{\alpha}(t-s;A)B(s)S_{\alpha,n-1}(s;A)x\mathrm{d}s\label{Salphan},\\
	&T_{\alpha}(t;A)x=\int\limits_{0}^{t}g_{\alpha-1}(t-s)C_{\alpha}(s;A)x\mathrm{d}s, \quad n\in \mathbb{N}, \quad x\in \mathbb{X}. \label{Talpha}
	\end{align}

\item With $K_{t}\coloneqq \sup\limits_{0\leq s \leq t}\Biggl\{ \norm{B(s)},\norm{B^{\prime}(s)}\Biggr\}$, we have for all $s\in [0,t]$, the bounds

\begin{align*}
&\norm{C_{\alpha}(s;A+B)} \leq Me^{\omega s}E_{\alpha}\left(MK_{t}s^{\alpha} \right),\\
&\norm{S_{\alpha}(s;A+B)} \leq Me^{\omega s}sE_{\alpha,2}\left(MK_{t}s^{\alpha} \right),\\
&\norm{C_{\alpha}(s;A+B)-C_{\alpha}(s;A)} \leq Me^{\omega s}\Big[E_{\alpha}\left(MK_{t}s^{\alpha} \right)-1\Big],\\
&\norm{S_{\alpha}(s;A+B)-S_{\alpha}(s;A)} \leq Me^{\omega s}s\Big[E_{\alpha,2}\left(MK_{t}s^{\alpha} \right)-1\Big]. 
\end{align*}
\end{itemize}
%Moreover, if $N=\max_{0\leq t\leq T}\|B(t)\|$, we have for all $t \in [0,T]$, the bounds
%\begin{align}\label{u}
%&\|u(t)\|\leq Me^{\omega t}E_{\alpha}(MNt^{\alpha})\|x\|+Mte^{\omega t}E_{\alpha}(MNt^{\alpha})\|y\|\\
%&\|u(t)-C_{\alpha}(t;A)x-S_{\alpha}(t;A)y\|\leq Me^{\omega t}\left( E_{\alpha}(MNt^{\alpha})-1\right) \|x\|+Mte^{\omega t}\left( E_{\alpha}(MNt^{\alpha})-1\right) \|y\|.\nonumber
%\end{align}
\end{theorem}
\begin{proof}
	For simplicity, we split our proof into several steps.
\begin{enumerate}
	\item From $\norm{C_{\alpha}(t;A)}\leq M e^{\omega t}$, it follows that $\norm{S_{\alpha}(t;A)}\leq Mte^{\omega t}$. By using the formula \eqref{Talpha}, we have 
	\begin{align}\label{A}
	\norm{T_{\alpha}(t;A)}&\leq \int\limits_{0}^{t}g_{\alpha-1}(t-s)\norm{C_{\alpha}(s;A)}\mathrm{d}s\leq M\int\limits_{0}^{t}\frac{(t-s)^{\alpha-2}}{\Gamma(\alpha-1)}e^{\omega s}\mathrm{d}s\nonumber\\&\leq Me^{\omega t}\int\limits_{0}^{t}\frac{(t-s)^{\alpha-2}}{\Gamma(\alpha-1)}\mathrm{d}s=Me^{\omega t}g_{\alpha}(t).
	\end{align}
	
	Put $K_{t}\coloneqq \sup\limits_{0\leq s \leq t}\Biggl\{\norm{B(s)}, \norm{B^{\prime}(s)}\Biggr\}$. Then, it is true that for $n \in N_{0}$
	
	\begin{align}\label{formula-1}
	&\norm{C_{\alpha,n}(t;A)}\leq M^{n+1} K_{t}^{n}e^{\omega t}g_{n\alpha+1}(t),\\
	&\norm{S_{\alpha,n}(t;A)}\leq M^{n+1} K_{t}^{n}e^{\omega t}g_{n\alpha+2}(t).
	\label{formula-2}
	\end{align}
	In the case of \eqref{formula-1}, this is true by our remark above for $n=0$. By mathematical induction principle, we verify 
	
	\begin{align}
	&\norm{C_{\alpha,n+1}(t;A)x}\leq \int\limits_{0}^{t}\norm{T_{\alpha}(t-s;A)}\norm{B(s)}\norm{C_{\alpha,n}(s;A)x} \mathrm{d}s\nonumber\\
	&\leq M^{n+2}K_{t}^{n+1}\norm{x}\int\limits_{0}^{t}g_{\alpha}(t-s)
	e^{w(t-s)}e^{\omega s}g_{n\alpha+1}(s)\mathrm{d}s\nonumber\\
	&\leq M^{n+2}K_{t}^{n+1}e^{\omega t}\norm{x}\int\limits_{0}^{t}\frac{(t-s)^{\alpha-1}}{\Gamma(\alpha)}\frac{s^{n\alpha}}{\Gamma(n\alpha+1)}\mathrm{d}s\nonumber\\
	&=M^{n+2}K_{t}^{n+1}e^{\omega t}\norm{x}g_{(n+1)\alpha+1}(t), \quad x \in \mathbb{X}.
	\end{align}
	
	The analogues procedure gives the bound for $S_{\alpha,n}(t;A)$, $t\in \mathbb{R_{+}}$, as follows:
	\allowdisplaybreaks
	\begin{align}
	&\norm{S_{\alpha,n+1}(t;A)x}\leq \int\limits_{0}^{t}\norm{T_{\alpha}(t-s;A)}\norm{B(s)}\norm{S_{\alpha,n}(s;A)x} \mathrm{d}s\nonumber\\
	&\leq M^{n+2}K_{t}^{n+1}\norm{x}\int\limits_{0}^{t}g_{\alpha}(t-s)
	e^{w(t-s)}e^{\omega s}g_{n\alpha+2}(s)\mathrm{d}s\nonumber\\
	&\leq M^{n+2}K_{t}^{n+1}e^{\omega t}\norm{x}\int\limits_{0}^{t}\frac{(t-s)^{\alpha-1}}{\Gamma(\alpha)}\frac{s^{n\alpha+1}}{\Gamma(n\alpha+2)}\mathrm{d}s\nonumber\\
	&=M^{n+2}K_{t}^{n+1}e^{\omega t}\norm{x}g_{(n+1)\alpha+2}(t), \quad x \in \mathbb{X}.
	\end{align}
	
	From these bounds it follows that the series representing $C_{\alpha}(t;A+B)$ and $S_{\alpha}(t;A+B)$ in \eqref{Cdefine} and \eqref{Sdefine}, respectively are uniformly convergent on compact subsets of $\mathbb{R}_{+}$ with respect to the operator norm topology. Hence, $C_{\alpha}(t;A+B)$ and $S_{\alpha}(t;A+B)$  are strongly continuous families on $\mathbb{R}_{+}$ with values in $\mathcal{L}(\mathbb{X})$.

	Furthermore, the bounds for $C_{\alpha}(s;A+B)$, $S_{\alpha}(s;A+B)$, $C_{\alpha}(s;A+B)-C_{\alpha}(s;A)$, and $S_{\alpha}(s;A+B)-S_{\alpha}(s;A)$, $s\in[0,t]$ derive as the following direct consequences:
	\begin{align*}
	\norm{C_{\alpha}(s;A+B)} \leq \sum_{n=0}^{\infty}\norm{C_{\alpha,n}(s;A)}\leq Me^{\omega s}\sum_{n=0}^{\infty}M^{n}K_{t}^{n}\frac{s^{n\alpha}}{\Gamma(n\alpha+1)}=Me^{\omega s}E_{\alpha}\left(MK_{t}s^{\alpha} \right), 
	\end{align*}
	
	\begin{align*}
	\norm{S_{\alpha}(s;A+B)} \leq \sum_{n=0}^{\infty}\norm{S_{\alpha,n}(s;A)}\leq Mse^{\omega s}\sum_{n=0}^{\infty}M^{n}K_{t}^{n}\frac{s^{n\alpha}}{\Gamma(n\alpha+2)}=Me^{\omega s}sE_{\alpha,2}\left(MK_{t}s^{\alpha} \right), 
	\end{align*}
	\begin{align*}
	\norm{C_{\alpha}(s;A+B)-C_{\alpha}(s;A)} \leq \sum_{n=1}^{\infty}\norm{C_{\alpha,n}(s;A)}\leq Me^{\omega s}\sum_{n=1}^{\infty}M^{n}K_{t}^{n}\frac{s^{n\alpha}}{\Gamma(n\alpha+1)}=Me^{\omega s}\Big[E_{\alpha}\left(MK_{t}s^{\alpha} \right)-1\Big], 
	\end{align*}
	\begin{align*}
	\norm{S_{\alpha}(s;A+B)-S_{\alpha}(s;A)} \leq \sum_{n=1}^{\infty}\norm{S_{\alpha,n}(s;A)}\leq Mse^{\omega s}\sum_{n=1}^{\infty}M^{n}K_{t}^{n}\frac{s^{n\alpha}}{\Gamma(n\alpha+2)}=Me^{\omega s}s\Big[E_{\alpha,2}\left(MK_{t}s^{\alpha} \right)-1\Big]. 
	\end{align*}
\item We show that, for $x\in \mathscr{D}(A)$, the function $t\mapsto C_{\alpha}(t;A+B)x$ is twice continuously differentiable. Assuming that this is true for $t\mapsto C_{\alpha,n}(t;A)x$, then as in Phillips \cite{philips}, it is easily shown to be true for $t\mapsto B(t)C_{\alpha,n}(t;A)x$ and thus, for $t\mapsto C_{\alpha,n+1}(t;A)x$. Furthermore, for $x\in \mathscr{D}(A)$, we have the following formulas:
\allowdisplaybreaks
\begin{align}
&C^{\prime}_{\alpha}(t;A)x=   \prescript{}{}{\mathcal{I}^{1}_{t}} \prescript{}{}{\mathcal{D}^{2}_{t}}C_{\alpha}(t;A)x=\prescript{}{}{\mathcal{I}^{\alpha-1}_{t}}\prescript{}{}{\mathcal{I}^{2-\alpha}_{t}} \prescript{}{}{\mathcal{D}^{2}_{t}}C_{\alpha}(t;A)x\nonumber\\
&\hspace{1.5cm}=\prescript{}{}{\mathcal{I}^{\alpha-1}_{t}}\prescript{C}{}{\mathcal{D}^{\alpha}_{t}}C_{\alpha}(t;A)x=\prescript{}{}{\mathcal{I}^{\alpha-1}_{t}}AC_{\alpha}(t;A)x=\prescript{}{}{\mathcal{I}^{\alpha-1}_{t}}C_{\alpha}(t;A)Ax=T_{\alpha}(t;A)Ax=AT_{\alpha}(t;A)x,\nonumber\\
&C^{\prime}_{\alpha,n+1}(t;A)x=\int\limits_{0}^{t}T^{\prime}_{\alpha}(t-s;A)B(s)C_{\alpha,n}(s;A)x\mathrm{d}s,\nonumber\\
&C^{\prime \prime}_{\alpha}(t;A)x=T^{\prime}_{\alpha}(t;A)Ax=AT^{\prime}_{\alpha}(t;A)x,\nonumber\\
&C^{\prime \prime}_{\alpha,n+1}(t;A)x=\int\limits_{0}^{t}T^{\prime}_{\alpha}(t-s;A)B^{\prime}(s)C_{\alpha,n}(s;A)x\mathrm{d}s+\int\limits_{0}^{t}T^{\prime}_{\alpha}(t-s;A)B(s)C^{\prime}_{\alpha,n}(s;A)x\mathrm{d}s,\nonumber\\
& T^{\prime}_{\alpha}(t;A)x=g_{\alpha-1}(t)\ast T_{\alpha}(t;A)Ax+g_{\alpha-1}(t)x=g_{\alpha-1}(t)\ast AT_{\alpha}(t;A)x+g_{\alpha-1}(t)x\label{Tder}.
\end{align}

Furthermore, from $\norm{T_{\alpha}(t;A)}\leq M e^{\omega t}g_{\alpha}(t)$, it follows that 
\begin{align}\label{talpha}
\norm{T^{\prime}_{\alpha}(t;A)x}&\leq \int\limits_{0}^{t}g_{\alpha-1}(t-s)\norm{T_{\alpha}(s;A)}\mathrm{d}s\norm{Ax}+g_{\alpha-1}(t)\|x\|\nonumber\\
&\leq M\int\limits_{0}^{t}g_{\alpha-1}(t-s)g_{\alpha}(s)e^{\omega s}\mathrm{d}s\norm{Ax}+g_{\alpha-1}(t)\|x\|\nonumber\\
&\leq Me^{\omega t}g_{2\alpha-1}(t)\norm{Ax}+g_{\alpha-1}(t)\norm{x}, \quad  x \in \mathscr{D}(A).
\end{align}

Using \eqref{formula-1} and \eqref{talpha}, we obtain by induction from these equations \eqref{Tder} the following bounds for $x \in \mathscr{D}(A)$:
\begin{align*}
&\|C^{\prime}_{\alpha}(t;A)x\|\leq Me^{\omega t}g_{\alpha}(t)\|Ax\|,\\
&\|C^{\prime}_{\alpha,1}(t;A)x\|\leq M^2 K_{t}e^{\omega t}g_{2\alpha}(t)\|Ax\|+MK_{t}e^{\omega t}g_{\alpha}(t)\norm{x},\\
&|C^{\prime}_{\alpha,n}(t;A)x\|\leq M^{n+1}K^{n}_{t}e^{\omega t}g_{n\alpha+\alpha}(t)\|Ax\|+M^{n}K^{n}_{t}e^{\omega t}g_{n\alpha}(t)\|x\|, \quad n\in \mathbb{N},\\\\
&\|C^{\prime\prime}_{\alpha}(t;A)x\|\leq Me^{\omega t}g_{2\alpha-1}(t)\|A^{2}x\|+g_{\alpha-1}(t)\|Ax\|,\\ 
&\|C^{\prime\prime}_{\alpha,1}(t;A)x\|\leq M K_{t}e^{\omega t}g_{\alpha}(t)\|x\|+
M K_{t}e^{\omega t}\Big( Mg_{2\alpha}(t)+g_{2\alpha-1}(t)\Big) \|Ax\|+
M^2 K_{t}e^{\omega t}g_{3\alpha-1}(t)\|A^{2}x\|,\\
&|C^{\prime\prime}_{\alpha,n}(t;A)x\|\leq M^{n}K^{n}_{t}e^{\omega t}\Big( Mg_{n\alpha+\alpha}(t)+ng_{n\alpha+\alpha-1}(t)\Big) \|Ax\|+M^{n+1}K^{n}_{t}e^{\omega t}g_{n\alpha+2\alpha-1}(t)\|A^{2}x\|\\
&\hspace{+2 cm}+M^{n-1}K^{n}_{t}e^{\omega t}\Big( Mg_{n\alpha}(t)+g_{n\alpha-1}(t)\Big) \|x\|,\quad n\geq 2.
\end{align*}
In a completely analogues manner $t\mapsto S_{\alpha}(t;A+B)$ is seen to be twice continuously differentiable on $\mathbb{R_{+}}$ for every $x\in \mathscr{D}(A)$, as follows:
\allowdisplaybreaks
\begin{align}
&S^{\prime}_{\alpha}(t;A)x=C_{\alpha}(t;A)x,\nonumber\\
&S^{\prime}_{\alpha,n+1}(t;A)x=\int\limits_{0}^{t}T^{\prime}_{\alpha}(t-s;A)B(s)S_{\alpha,n}(s;A)x\mathrm{d}s,\nonumber\\
&S^{\prime\prime}_{\alpha}(t;A)x=T_{\alpha}(t;A)Ax=AT_{\alpha}(t;A)x,\nonumber\\
&S^{\prime \prime}_{\alpha,n+1}(t;A)x=\int\limits_{0}^{t}T^{\prime}_{\alpha}(t-s;A)B^{\prime}(s)S_{\alpha,n}(s;A)x\mathrm{d}s\nonumber\\
&\hspace{2cm}+\int\limits_{0}^{t}T^{\prime}_{\alpha}(t-s;A)B(s)S^{\prime}_{\alpha,n}(s;A)x\mathrm{d}s\label{Sder}.
\end{align}
Using \eqref{formula-2} and \eqref{talpha}, we obtain by induction from these equations \eqref{Sder} the following bounds:
\begin{align*}
&\|S^{\prime}_{\alpha}(t;A)x\|\leq Me^{\omega t}\|x\|,\\
&\|S^{\prime}_{\alpha,1}(t;A)x\|\leq M^2 K_{t}e^{\omega t}g_{2\alpha+1}(t)\|Ax\|+ M K_{t}e^{\omega t}g_{\alpha+1}(t)\|x\|,\\
&|S^{\prime}_{\alpha,n}(t;A)x\|\leq M^{n+1}K^{n}_{t}e^{\omega t}g_{n\alpha+\alpha+1}(t)\|Ax\|+M^{n}K^{n}_{t}e^{\omega t}g_{n\alpha+1}(t)\|x\|,\quad n\in \mathbb{N},\\\\
&\|S^{\prime\prime}_{\alpha}(t;A)x\|\leq Me^{\omega t}g_{\alpha}(t)\|Ax\|,\\ 
&\|S^{\prime\prime}_{\alpha,1}(t;A)x\|\leq M^2 K_{t}e^{\omega t}\Big( g_{2\alpha}(t)+g_{2\alpha+1}(t)\Big) \|Ax\|+M K_{t}e^{\omega t}\Big( g_{2\alpha}(t)+g_{\alpha+1}(t)\Big) \|x\|,\\
&\|S^{\prime\prime}_{\alpha,n}(t;A)x\|\leq M^{n}K^{n}_{t}e^{\omega t}\Big( Mg_{n\alpha+\alpha+1}(t)+ng_{n\alpha+\alpha}(t)\Big)\|Ax\|+M^{n+1}K^{n}_{t}e^{\omega t}g_{n\alpha+2\alpha}(t)\|A^{2}x\|\\
&\hspace{+2 cm}+M^{n-1}K^{n}_{t}e^{\omega t}\Big( g_{n\alpha}(t)+Mg_{n\alpha+1}(t)\Big)\|x\|, \quad n \geq 2.
\end{align*}

Therefore, for $x \in \mathscr{D}(A)$, the series $\sum\limits_{n=0}^{\infty}C^{\prime}_{\alpha,n}(t;A)x$, $\sum\limits_{n=0}^{\infty}C^{\prime \prime}_{\alpha,n}(t;A)x$ and $\sum\limits_{n=0}^{\infty}S^{\prime}_{\alpha,n}(t;A)x$, $\sum\limits_{n=0}^{\infty}S^{\prime \prime}_{\alpha,n}(t;A)x$ converge uniformly in every compact interval of $\mathbb{R_{+}}$ to continuous functions which are $C^{\prime}_{\alpha}(t;A)x$, $C^{\prime \prime}_{\alpha}(t;A)x$ and $S^{\prime}_{\alpha}(t;A)x$, $S^{\prime \prime}_{\alpha}(t;A)x$, respectively.

\item Next we prove that $u\in \mathscr{D}(A)$ with initial conditions $u(0)=x$, $u^{\prime}(0)=y$, where $x,y\in \mathscr{D}(A)$ satisfies the abstract Cauchy problem \eqref{eq-3}. Since 
\begin{align*}
&C_{\alpha}(0;A)=I,\quad C_{\alpha,n}(0;A)=0, \quad C^{\prime}_{\alpha}(0;A)=0, \quad C^{\prime}_{\alpha,n}(0;A)=0,\quad  n\in\mathbb{N},\\
&S_{\alpha}(0;A)=0, \quad S_{\alpha,n}(0;A)=0,\quad S^{\prime}_{\alpha}(0;A)=I, \quad S^{\prime}_{\alpha,n}(0;A)=0,\quad  n\in\mathbb{N},
\end{align*}
we have \eqref{csder} and \eqref{ssder}, i.e. initial conditions \eqref{initial} are satisfied. Applying \eqref{usol}, \eqref{Cdefine}-\eqref{Salphan}, it follows that
\begin{align}\label{finu}
u(t)&=C_{\alpha}(t;A)x+S_{\alpha}(t;A)y+\sum_{n=1}^{\infty}\int\limits_{0}^{t}T_{\alpha}(t-s;A)B(s)\Big(C_{\alpha,n-1}(s;A)x+ S_{\alpha,n-1}(s;A)y\Big)\mathrm{d}s\nonumber\\
&=C_{\alpha}(t;A)x+S_{\alpha}(t;A)y+\sum_{n=0}^{\infty}\int\limits_{0}^{t}T_{\alpha}(t-s;A)B(s)\Big(C_{\alpha,n}(s;A)x+ S_{\alpha,n}(s;A)y\Big)\mathrm{d}s\nonumber\\
&=C_{\alpha}(t;A)x+S_{\alpha}(t;A)y+\int\limits_{0}^{t}T_{\alpha}(t-s;A)B(s)u(s)\mathrm{d}s,
\end{align}
where the interchanging of the summation and integration is justified by the uniform convergence of the series. Then integro-differentiating \eqref{finu} term-by-term and by using the formulas \eqref{AA}-\eqref{2.26}, we obtain for $x,y \in\mathscr{D}(A)$:

\begin{align}\label{3.17}
 \prescript{C}{}{\mathcal{D}^{\alpha}_{t}}C_{\alpha}(t;A)x&=\prescript{}{}{\mathcal{I}^{2-\alpha}_{t}}{\mathcal{D}}_{t}^{2}C_{\alpha}(t;A)x=A\prescript{}{}{\mathcal{I}^{2-\alpha}_{t}}{\mathcal{D}}_{t}^{1}T_{\alpha}(t;A)x\nonumber \\
 &=
 A\prescript{}{}{\mathcal{I}^{2-\alpha}_{t}}{\mathcal{D}}_{t}^{1}\left( {\mathcal{I}}_{t}^{\alpha-1}C_{\alpha}(t;A)x\right)=A\prescript{}{}{\mathcal{I}^{2-\alpha}_{t}}\Big( {\mathcal{I}}_{t}^{\alpha-1}AT_{\alpha}(t;A)x+g_{\alpha-1}(t)x\Big)\nonumber\\
 &=A\Big(\prescript{}{}{\mathcal{I}^{1}_{t}}AT_{\alpha}(t;A)x+x\Big)=A\Big(A\prescript{}{}{\mathcal{I}^{1}_{t}}T_{\alpha}(t;A)x+x\Big)=AC_{\alpha}(t;A)x,
\end{align}

\begin{align}
\prescript{C}{}{\mathcal{D}^{\alpha}_{t}}S_{\alpha}(t;A)y&=\prescript{}{}{\mathcal{I}^{2-\alpha}_{t}}{\mathcal{D}}_{t}^{2}S_{\alpha}(t;A)y=\prescript{}{}{\mathcal{I}^{2-\alpha}_{t}}{\mathcal{D}}_{t}^{1}C_{\alpha}(t;A)y\nonumber \\
&=
A\prescript{}{}{\mathcal{I}^{2-\alpha}_{t}} T_{\alpha}(t;A)x=A\prescript{}{}{\mathcal{I}^{2-\alpha}_{t}} {\mathcal{I}}_{t}^{\alpha-1}C_{\alpha}(t;A)y\nonumber\\
&=A\prescript{}{}{\mathcal{I}^{1}_{t}}C_{\alpha}(t;A)x=AS_{\alpha}(t;A)y.
\end{align}
Then, it follows that
\begin{align}\label{Du}
\left( \prescript{C}{}{\mathcal{D}^{\alpha}_{t}}u\right) (t)=AC_{\alpha}(t;A)x+AS_{\alpha}(t;A)y+ \prescript{C}{}{\mathcal{D}^{\alpha}_{t}}\left( \int\limits_{0}^{t}T_{\alpha}(t-s;A)B(s)u(s)\mathrm{d}s\right).
\end{align}
Making use of \eqref{Du} and \eqref{RLandC}, $\left( T_{\alpha}\ast (Bu)\right) (0)=\left( T_{\alpha}\ast (Bu)\right)^{\prime}(0)=0$, the property of $\mathcal{I}^{\alpha}_{t}\left( f\ast g\right) =\left( \mathcal{I}^{\alpha}_{t}f\right) \ast g$ and the semigroup property for operators of Riemann-Liouville fractional integration, we have

\begin{align}\label{Dc}
\prescript{C}{}{\mathcal{D}^{\alpha}_{t}}\Big( T_{\alpha}(t;A)\ast \Big( B(t)u(t)\Big) \Big)&= \prescript{}{}{\mathcal{D}^{\alpha}_{t}}\Big( T_{\alpha}(t;A)\ast \Big( B(t)u(t)\Big) \Big)= \prescript{}{}{\mathcal{D}^{2}_{t}}\mathcal{I}^{2-\alpha}_{t}\Big(\mathcal{I}^{\alpha-1}_{t} C_{\alpha}(t;A)\ast \Big( B(t)u(t)\Big) \Big)\nonumber\\
&=\prescript{}{}{\mathcal{D}^{2}_{t}}\Big(\mathcal{I}^{2-\alpha}_{t}\mathcal{I}^{\alpha-1}_{t} C_{\alpha}(t;A)\Big)\ast \Big( B(t)u(t)\Big)= \prescript{}{}{\mathcal{D}^{1}_{t}}\Big( C_{\alpha}(t;A)\ast \Big( B(t)u(t)\Big) \Big)\nonumber\\&=\Big( \prescript{}{}{\mathcal{D}^{1}_{t}} C_{\alpha}(t;A)\Big) \ast \Big( B(t)u(t)\Big) +B(t)u(t),
\end{align}
where
\begin{align}\label{D}
\prescript{}{}{\mathcal{D}^{1}_{t}} C_{\alpha}(t;A)=T_{\alpha}(t;A)A=AT_{\alpha}(t;A),
\end{align}
which together with \eqref{Du}, \eqref{Dc}, and using the closedness of $A$ implies that $u(t)$ satisfies  \eqref{eq-3} for any $t \in \mathbb{R_{+}}$.
\end{enumerate}
 
 \textit{Uniqueness:} To prove the uniqueness, we let $v:\mathbb{R}_{+}\to \mathscr{D}(A)$ be a solution of \eqref{eq-3} with zero initial conditions $v(0)=v^{\prime}(0)=0$. Then, by using the well-known property \eqref{gcon}, we have $v(t)= \prescript{}{}{\mathcal{I}^{\alpha}_{t}}Av(t)+\prescript{}{}{\mathcal{I}^{\alpha}_{t}}\left( B(t)v(t)\right) $ and applying the variation of parameters formula, $v(t)$ satisfies the Volterra integral equation of second-kind:
 \begin{equation*}
 v(t)=\int\limits_{0}^{t}T_{\alpha}(t-s;A)B(s)v(s)\mathrm{d}s.
 \end{equation*}
 Setting $m_{t}=\sup\limits_{0\leq s\leq t}\|v(s)\|$, we see that for $m_{t}>0$
 \begin{equation*}
 m_{t}\leq \frac{MK_{t}m_{t}}{\Gamma(\alpha)}\int\limits_{0}^{t}(t-s)^{\alpha-1}e^{w(t-s)}\mathrm{d}s\leq MK_{t}m_t g_{\alpha+1}(t)e^{\omega t}< m_{t},
 \end{equation*}
 if $t>0$ is chosen sufficiently small. Thus, $v(t)=0$ on $[0,t_{0}]$ with $t_{0}>0$. Iteration of this argument leads to $v(t)\equiv 0$ for any $t \in \mathbb{R}_{+}$. The proof is complete.
\end{proof}
 
 %===========================================================
 
 The following lemma plays a crucial role in the proof of Theorem 3.2. 
 \begin{lemma}\label{lemma-2}
Let $A$ be infinitesimal generator of the strongly continuous fractional cosine family $\left\lbrace C_{\alpha}(t;A), t \in \mathbb{R_{+}}\right\rbrace $ and $ T_{\alpha}(\cdot;A):\mathbb{R_{+}}\to \mathcal{L}(\mathbb{X})$ be the corresponding fractional Riemann-Liouville family  on $\mathbb{X}$. For $f(t)$ strongly continuous on $\mathbb{R_{+}}$ to $\mathbb{X}$, $g(t) = \int\limits_{0}^{t}T_{\alpha}(t-s;A)f(s)\mathrm{d}s=\int\limits_{0}^{t}T_{\alpha}(s;A)f(t-s)\mathrm{d}s$ exists and is itself strongly continuous on $\mathbb{R_{+}}$ to $\mathbb{X}$.
 If $f(t)$ is strongly continuously differentiable on $\mathbb{R_{+}}$ to $\mathbb{X}$, then  $g(t)$ is  strongly continuously differentiable for any $t\in\mathbb{R_{+}}$ and
	\begin{align}
	\mathcal{D}_{t}^{1}g(t)&=T_{\alpha}(t;A)f(0)+\int\limits_{0}^{t}T_{\alpha}(t-s;A)f^{\prime}(s)\mathrm{d}s\nonumber\\&=\int\limits_{0}^{t}T^{\prime}_{\alpha}(t-s;A)f(s)\mathrm{d}s, \quad t \in \mathbb{R_{+}},
	\end{align}
where $T^{\prime}_{\alpha}(t;A)$ is defined as in \eqref{Tder}.
 \end{lemma}
 \begin{proof}
Since $\norm{T_{\alpha}(t;A)}\leq M e^{\omega t}g_{\alpha}(t)$ for any $t\in \mathbb{R_{+}}$, it is obvious that $T_{\alpha}(t-s;A)f(s)$ is strongly continuous in $s\in[0,t]$ whenever the same is true of $f(s)$. In this case, $g(t)=\int\limits_{0}^{t}T_{\alpha}(t-s;A)f(s)\mathrm{d}s$ will exist in the strong topology and be equal to $\int\limits_{0}^{t}T_{\alpha}(s;A)f(t-s)\mathrm{d}s$. Moreover, since $f:\mathbb{R_{+}}\to\mathcal{L}(\mathbb{X})$ is strongly continuously differentiable for any $t\in\mathbb{R_{+}}$, by using the well-known integration by parts formula for the operator-valued functions, we obtain the desired result:
\begin{align*}
\int\limits_{0}^{t}T_{\alpha}(t-s;A)f^{\prime}(s)\mathrm{d}s=-T_{\alpha}(t;A)f(0)+\int\limits_{0}^{t}T^{\prime}_{\alpha}(t-s;A)f(s)\mathrm{d}s, \quad t \in \mathbb{R_{+}}.
\end{align*}
 \end{proof}
 \begin{theorem}\label{thm3.2}
 	Let $\left\lbrace C_{\alpha}(t;A): t\in \mathbb{R_{+}}\right\rbrace $ be a strongly continuous fractional cosine family with infinitesimal generator $A$ and $\left\lbrace T_{\alpha}(t;A): t\in \mathbb{R_{+}}\right\rbrace $ be fractional Riemann-Liouville family corresponding to  the $C_{\alpha}$. Let $B(t)$ and $f(t)$ be strongly continuously differentiable functions on $\mathbb{R_{+}}$ to $\mathcal{L}(\mathbb{X})$ and $\mathbb{X}$, respectively.
 	Then there exists a unique $2$-times continuously differentiable function $w:\mathbb{R_{+}}
 	\to \mathscr{D}(A)$ which is a particular solution of equation \eqref{eq-2}
 	with zero initial conditions $w(0)=w^{\prime}(0)=0$.
 	This solution has the closed-form
 	\begin{equation}
 	w(t)=\sum_{n=0}^{\infty}w_{n}(t),
 	\end{equation}
 	where 
 	\allowdisplaybreaks
 	\begin{align*}
 	&w_{0}(t)=\int\limits_{0}^{t}T_{\alpha}(t-s;A)f(s)\mathrm{d}s,\\
 	&w_{n}(t)=\int\limits_{0}^{t}T_{\alpha}(t-s;A)B(s)w_{n-1}(s)\mathrm{d}s, \quad t >0, \quad n \in \mathbb{N}.
 	\end{align*}
 \end{theorem}
 
 \begin{proof}
 	It follows since $f\in \mathbb{C}
 	^{1}(\mathbb{R_{+}},\mathbb{X})$, $w_{0}(t)$ is $2$-times strongly continuously differentiable and hence by induction that $w_{n}(t)$ is likewise for any $n \in \mathbb{N}$. In fact, by Lemma \ref{lem3.1}, we have
\allowdisplaybreaks
 	\begin{align*}
 	&w^{\prime}_{0}(t)=\int\limits_{0}^{t}T_{\alpha}^{\prime}(t-s;A)f(s)\mathrm{d}s,\\
 	&w^{\prime \prime}_{0}(t)=T^{\prime}_{\alpha}(t;A)f(0)+\int\limits_{0}^{t}T_{\alpha}^{\prime}(t-s;A)f^{\prime }(s)\mathrm{d}s,\\\\
 	&w^{\prime}_{n}(t)=\int\limits_{0}^{t}T_{\alpha}^{\prime}(t-s;A)B(s)w_{n-1}(s)\mathrm{d}s,\\
 	&w^{\prime \prime}_{n}(t)=\int\limits_{0}^{t}T_{\alpha}^{\prime}(t-s;A)B^{\prime}(s)w_{n-1}(s)\mathrm{d}s+\int\limits_{0}^{t}T_{\alpha}^{\prime}(t-s;A)B(s)w^{\prime }_{n-1}(s)\mathrm{d}s, \quad n \in \mathbb{N}.
 	\end{align*}
 By using \eqref{A} and \eqref{talpha} it is easy to acquire the estimations for $x\in \mathscr{D}(A)$:
 	\begin{align*}
 	&\|w_{n}(t)x\|\leq M^{n+1}K_{t}^{n}N_{t}e^{\omega t}g_{n\alpha+\alpha+1}(t)\norm{x}, \quad n \in \mathbb{N}_{0},\\\\
 	&\|w_{0}^{\prime}(t)x\|\leq MN_{t}e^{\omega t}g_{2\alpha}(t)\norm{Ax}+N_{t}g_{\alpha}(t)\norm{x},\\
 	&\|w_{1}^{\prime}(t)x\|\leq M^{2}K_{t}N_{t}e^{\omega t}g_{3\alpha}(t)\norm{Ax}+MK_{t}N_{t}e^{\omega t}g_{2\alpha}(t)\norm{x},\\
 	&\|w_{n}^{\prime}(t)x\|\leq M^{n+1}K_{t}^{n}N_{t}e^{\omega t}g_{n\alpha+2\alpha}(t)\norm{Ax}+M^{n}K_{t}^{n}N_{t}e^{\omega t}g_{n\alpha+\alpha}(t)\norm{x},\quad n \in \mathbb{N},\\\\
 	&\|w_{0}^{\prime \prime}(t)x\|\leq MN_{t}e^{\omega t}\Big[g_{2\alpha-1}(t)+g_{2\alpha}(t)\Big]\norm{Ax}+N_{t}\Big[g_{\alpha-1}(t)+g_{\alpha}(t)\Big]\norm{x},\\
 	&\|w_{1}^{\prime \prime}(t)x\|\leq M^{2}K_{t}N_{t}e^{\omega t}g_{4\alpha-1}(t)\norm{A^{2}x}+MK_{t}N_{t}e^{\omega t}\Big[2g_{3\alpha-1}(t)+	Mg_{3\alpha}(t)\Big]\norm{Ax}\\&\hspace{+1.45 cm}+K_{t}N_{t}\Big[Me^{\omega t}g_{2\alpha}(t)+g_{2\alpha-1}(t)\Big]\norm{x},\\
 	&\|w_{n}^{\prime \prime}(t)\|\leq M^{n+1}K_{t}^{n}N_{t}e^{\omega t}g_{n\alpha+3\alpha-1}(t)\norm{A^{2}x}+M^{n}K_{t}^{n}N_{t}e^{\omega t}\Big[(n+1)g_{n\alpha+2\alpha-1}(t)+Mg_{n\alpha+2\alpha}(t)\Big]\norm{Ax}\\&\hspace{+1.4 cm}+K_{t}^{n}N_{t}\Big[M^{n}e^{\omega t}g_{n\alpha+\alpha}(t)+g_{n\alpha+\alpha-1}(t)\Big]\norm{x}, \quad n \in \mathbb{N},
 	\end{align*}
 	where $N_{t}\coloneqq \sup\limits_{0 \leq s \leq t}\Biggl\{\norm{f(s)}, \norm{f^{\prime}(s)}\Biggr\}$.
 	
 	If we set $w(t)=\sum\limits_{n=0}^{\infty}w_{n}(t)$, then as in Theorem \ref{thm1}, $w(t)$ is 2-times strongly continuously differentiable on $\mathbb{R_{+}}$ to $\mathscr{D}(A)$ and $w^{\prime}(t)=\sum\limits_{n=0}^{\infty}w^{\prime}_{n}(t)$, $w^{\prime \prime}(t)=\sum\limits_{n=0}^{\infty}w^{\prime \prime}_{n}(t)$. Furthermore, $w(0)=w^{\prime}(0)=0$. From the definition of $w_{n}(t)$, $n\in \mathbb{N}_{0}$ and uniform convergence of the series $\sum\limits_{n=0}^{\infty}w_{n}(t)$ in every finite interval of $\mathbb{R_{+}}$, it follows that
 	\begin{align}
 		w(t)=\sum_{n=0}^{\infty}w_{n}(t)=\sum_{n=1}^{\infty}w_{n}(t)+w_{0}(t)&=w_{0}(t)+\sum_{n=1}^{\infty}\int\limits_{0}^{t}T_{\alpha}(t-s;A)B(s)w_{n-1}(s)\mathrm{d}s\nonumber\\
 		&=w_{0}(t)+\int\limits_{0}^{t}T_{\alpha}(t-s;A)B(s)\sum_{n=1}^{\infty}w_{n-1}(s)\mathrm{d}s\nonumber\\
 		&=w_{0}(t)+\int\limits_{0}^{t}T_{\alpha}(t-s;A)B(s)\sum_{n=0}^{\infty}w_{n}(s)\mathrm{d}s\nonumber\\
 		&=w_{0}(t)+\int\limits_{0}^{t}T_{\alpha}(t-s;A)B(s)w(s)\mathrm{d}s, \label{special1}
 	\end{align}
 where the interchanging summation and integration is justified by the uniform convergence of the series.
 	 Since $w(t)$ is $2$-times strongly continuously differentiable, we may differentiate \eqref{special1} term-wise in Caputo's sense by considering $C_{\alpha}(0;A)=I$, using \eqref{RLandC} and \eqref{Dc} as follows:
 	 \begin{align*}
 	 \prescript{C}{}{\mathcal{D}^{\alpha}_{t}}w(t)&=\prescript{C}{}{\mathcal{D}^{\alpha}_{t}}\Big(w_{0}(t)+\int\limits_{0}^{t}T_{\alpha}(t-s;A)B(s)w(s)\mathrm{d}s\Big)\\
 	 &=\prescript{C}{}{\mathcal{D}^{\alpha}_{t}}\Big(T_{\alpha}(t;A) \ast f(t)\Big)+ \prescript{C}{}{\mathcal{D}^{\alpha}_{t}}\Big(T_{\alpha}(t;A) \ast\Big(B(t)w(t)\Big)\Big)\\
 	 &=\prescript{}{}{\mathcal{D}^{\alpha}_{t}}\Big(T_{\alpha}(t;A) \ast f(t)\Big)+ \prescript{}{}{\mathcal{D}^{\alpha}_{t}}\Big(T_{\alpha}(t;A) \ast\Big(B(t)w(t)\Big)\Big)\\
 	 &=AT_{\alpha}(t;A)\ast f(t)+f(t)+AT_{\alpha}(t;A)\ast \Big(B(t)w(t)\Big)+B(t)w(t)\\&=
 	 Aw_{0}(t)+f(t)+A\int\limits_{0}^{t}T_{\alpha}(t-s;A)B(s)w(s)\mathrm{d}s+B(t)w(t)\\
 	  &=\Big(A+B(t) \Big)w(t)+f(t). 
 	 \end{align*}
 	Note that since initial values are zero we have used Riemann-Liouville fractional derivative instead of Caputo one. This shows immediately that $w(t)$ is a particular solution of \eqref{eq-2}. The uniqueness of solution follows precisely as in the \textit{uniqueness} proof of Theorem \ref{thm1}. The proof is complete.
 \end{proof}
	\begin{remark}
	Considering the homogeneous equation with the second initial condition equal to zero, we derive the special case published in \cite{bazhlekova2} by Bazhlekova. The reason is that in some models (e.g. for waves in viscoelastic media) there are some physical arguments to choose the second initial condition equal to zero in her thesis \cite{bazhlekova}. Therefore, our theory generalizes those results.
\end{remark}
\begin{itemize}
	\item  In this special case, if $B(t)\equiv B\in \mathcal{L}(\mathbb{X})$, then for any $x,y\in\mathscr{D}(A)$, the solution to the following abstract Cauchy problem for fractional evolution equation
in a Banach space $\mathbb{X}$:
\begin{align}\label{ABeq}
\begin{cases}
\prescript{C}{}{\mathcal{D}^{\alpha}_{t}}u(t)=(A+B)u(t)+f(t), \quad t>0,\\
u(0)=x,\quad u^{\prime}(0)=y,
\end{cases}
\end{align}
can be written as follows:
\begin{align*}
u(t)=C_{\alpha}(t;A+B)x+S_{\alpha}(t;A+B)y+w(t), \quad t>0,
\end{align*}
where
\allowdisplaybreaks
\begin{align*}
&C_{\alpha}(t;A+B)x=\sum_{n=0}^{\infty}C_{\alpha,n}(t,A)x,\\
&S_{\alpha}(t;A+B)x=\sum_{n=0}^{\infty}S_{\alpha,n}(t;A)x,\quad x\in \mathbb{X},\\
&w(t)=\sum_{n=0}^{\infty}w_{n}(t),
\end{align*} 
with
\allowdisplaybreaks
\begin{align*}
&C_{\alpha,0}(t;A)x=C_{\alpha}(t;A)x,\quad S_{\alpha,0}(t;A)x=S_{\alpha}(t;A)x, \quad w_{0}(t)=\int\limits_{0}^{t}T_{\alpha}(t-s)f(s)\mathrm{d}s,\nonumber\\
&C_{\alpha,n}(t;A)x= \int\limits_{0}^{t}T_{\alpha}(t-s;A)BC_{\alpha,n-1}(s;A)x\mathrm{d}s,\\
&S_{\alpha,n}(t;A)x= \int\limits_{0}^{t}T_{\alpha}(t-s;A)BS_{\alpha,n-1}(s;A)x\mathrm{d}s,\\
&w_{n}(t)=\int\limits_{0}^{t}T_{\alpha}(t-s)Bw_{n-1}(s)\mathrm{d}s,\quad n\in \mathbb{N},\quad x\in \mathbb{X}.
\end{align*}
In this case, the particular solution $w(t)$ of \eqref{ABeq} can also be put in a more suggestive form by stating and proving the following theorem.

\begin{theorem}
	Let $A$ be infinitesimal generator of $2$-times strongly continuous fractional family of linear operators $C_{\alpha}(t;A)$ on $\mathbb{R_{+}}$. Let $B\in \mathcal{L}(\mathbb{X})$ and $f(t)$ be strongly continuously differentiable function on $\mathbb{R_{+}}$ to $\mathbb{X}$. Then there exists a unique $2$-times continuously differentiable function $w:\mathbb{R_{+}}\to \mathscr{D}(A)$ which is a particular solution of equation \eqref{ABeq} with zero initial conditions $w(0)=w^{\prime}(0)=0$.
This solution has a variation of constants formula
\begin{equation}\label{a}
	w(t)=\int\limits_{0}^{t}T_{\alpha}(t-s;A+B)f(s)\mathrm{d}s, \quad t>0,
\end{equation}
where 
\begin{equation}\label{TALPHA}
T_{\alpha}(t;A+B)=\int\limits_{0}^{t}g_{\alpha-1}(t-s)C_{\alpha}(s;A+B)\mathrm{d}s.
\end{equation}
\end{theorem}
\begin{proof}
	From $\norm{C_{\alpha}(t;A+B)}\leq M e^{\omega t} E_{\alpha}\left(M\norm{B} t^{\alpha}\right)$ and \eqref{TALPHA}, we have 
	\allowdisplaybreaks
	\begin{align}
	\norm{T_{\alpha}(t;A+B)}&\leq \int\limits_{0}^{t}g_{\alpha-1}(t-s)\norm{C_{\alpha}(s;A+B)}\mathrm{d}s\nonumber\\
	&\leq M\int\limits_{0}^{t}\frac{(t-s)^{\alpha-2}}{\Gamma(\alpha-1)}e^{\omega s} E_{\alpha}\left(M\norm{B} s^{\alpha}\right)\mathrm{d}s\nonumber\nonumber\\&\leq e^{\omega t}\sum_{n=0}^{\infty}M^{n+1}\norm{B}^{n}\int\limits_{0}^{t}\frac{(t-s)^{\alpha-2}}{\Gamma(\alpha-1)}\frac{ s^{n\alpha}}{\Gamma(n\alpha+1)}\mathrm{d}s\nonumber\\
	&= e^{\omega t}\sum_{n=0}^{\infty}M^{n+1}\norm{B}^{n}\frac{t^{n\alpha+\alpha-1}}{\Gamma(n\alpha+\alpha)}\nonumber\\
	&=Me^{\omega t}t^{\alpha-1}E_{\alpha,\alpha}\left(M\norm{B} t^{\alpha}\right).\label{formula-T}
%	&=\sum_{n=0}^{\infty} M^{n+1}K_{t}^{n}e^{\omega t}g_{n\alpha+\alpha}(t).
	 \end{align}
 
In fact, by Lemma \ref{lemma-2} and formula \eqref{Tder}, we have
\allowdisplaybreaks
\begin{align*}
&w^{\prime}(t)=\int\limits_{0}^{t}T_{\alpha}(t-s;A+B)f^{\prime}(s)\mathrm{d}s+T_{\alpha}(t;A+B)f(0),\\
&w^{\prime \prime}(t)=T^{\prime}_{\alpha}(t;A+B)f(0)+\int\limits_{0}^{t}T^{\prime}_{\alpha}(t-s;A+B)f^{\prime}(s)\mathrm{d}s,\\
&T^{\prime}_{\alpha}(t;A+B)=g_{\alpha-1}(t)\ast C^{\prime}_{\alpha}(t;A+B)+g_{\alpha-1}(t),\\
&C^{\prime}_{\alpha}(t;A+B)=AT_{\alpha}(t;A)+\int\limits_{0}^{t}T^{\prime}_{\alpha}(t-s;A)BC_{\alpha}(s;A+B)\mathrm{d}s.
\end{align*}
It is now easy to acquire the estimations for $n\in \mathbb{N}_{0}$:
\begin{align*}
&\|w(t)x\|\leq MN_{t}e^{\omega t}t^{\alpha}E_{\alpha,\alpha+1}(M\norm{B}t^{\alpha})\norm{x},\\
&\|w^{\prime}(t)x\|\leq MN_{t}e^{\omega t}\Big[t^{\alpha-1}E_{\alpha,\alpha}(M\norm{B}t^{\alpha})+t^{\alpha}E_{\alpha,\alpha+1}(M\norm{B}t^{\alpha})\Big]\norm{x},\\
&\|w^{\prime \prime}(t)x\|\leq MN_{t}e^{\omega t}\Big[g_{2\alpha-1}(t)+g_{2\alpha}(t)\Big]\norm{Ax}+N_{t}\Big[g_{\alpha-1}(t)+g_{\alpha}(t)\Big]\norm{x}\\
&+M^{2}\norm{B}N_{t}e^{\omega t}\Big[t^{3\alpha-1}E_{\alpha,3\alpha}(M\norm{B}t^{\alpha})+t^{3\alpha-2}E_{\alpha,3\alpha-1}(M\norm{B}t^{\alpha})\Big]\norm{Ax}\\
&+M\norm{B}N_{t}e^{\omega t}\Big[t^{2\alpha-1}E_{\alpha,2\alpha}(M\norm{B}t^{\alpha})+t^{2\alpha-2}E_{\alpha,2\alpha-1}(M\norm{B}t^{\alpha})\Big]\norm{x},\quad x \in \mathscr{D}(A),
\end{align*}
where we have used the following estimations:
\begin{align*}
\|C_{\alpha}^{\prime}(t;A+B)x\|&\leq M^{2}\norm{B}e^{\omega t}t^{2\alpha-1}E_{\alpha,2\alpha}(M\norm{B}t^{\alpha})\norm{Ax}\\&+M\norm{B}e^{\omega t}t^{\alpha-1}E_{\alpha,\alpha}(M\norm{B}t^{\alpha})\norm{x}+Me^{\omega t}g_{\alpha}(t)\norm{Ax},
\end{align*}
\begin{align*}
&\|T_{\alpha}^{\prime}(t;A+B)x\|\leq M^{2}\norm{B}e^{\omega t}t^{3\alpha-2}E_{\alpha,3\alpha-1}(M\norm{B}t^{\alpha})\norm{Ax}\\&+M\norm{B}e^{\omega t}t^{2\alpha-2}E_{\alpha,2\alpha-1}(M\norm{B}t^{\alpha})\norm{x}+Me^{\omega t}g_{2\alpha-1}(t)\norm{Ax}+g_{\alpha-1}(t)\norm{x}, \quad x \in \mathscr{D}(A),
\end{align*}

with  $N_{t}\coloneqq \sup\limits_{0 \leq s \leq t}\Biggl\{\norm{f(s)}, \norm{f^{\prime}(s)}\Biggr\}$.

Therefore, $w(t)$ is 2-times strongly continuously differentiable function on $\mathbb{R_{+}}$.
Next, we prove that $w(t)$ satisfies \eqref{ABeq} with $w(0)=w^{\prime}(0)=0$. By making use of the formulas \eqref{RLandC},\eqref{a} and \eqref{TALPHA}, we derive that

	\begin{align*}
	\prescript{C}{}{\mathcal{D}^{\alpha}_{t}}w(t)&=\prescript{C}{}{\mathcal{D}^{\alpha}_{t}}\Big(\int\limits_{0}^{t}T_{\alpha}(t-s;A+B)f(s)\mathrm{d}s\Big)=\prescript{C}{}{\mathcal{D}^{\alpha}_{t}}\Big(T_{\alpha}(t;A+B) \ast f(t)\Big) =\prescript{}{}{\mathcal{D}^{\alpha}_{t}}\Big(T_{\alpha}(t;A+B) \ast f(t)\Big)\\
	&=\prescript{}{}{\mathcal{D}^{2}_{t}}\prescript{}{}{\mathcal{I}^{2-\alpha}_{t}}\Big( \prescript{}{}{\mathcal{I}^{\alpha-1}_{t}}C_{\alpha}(t;A+B)\Big)\ast f(t) 
	=\prescript{}{}{\mathcal{D}^{2}_{t}}\prescript{}{}{\mathcal{I}^{2-\alpha}_{t}} \prescript{}{}{\mathcal{I}^{\alpha-1}_{t}}\Big(C_{\alpha}(t;A+B)\ast f(t) \Big)\\
	&=\prescript{}{}{\mathcal{D}^{2}_{t}}\prescript{}{}{\mathcal{I}^{1}_{t}}\Big(C_{\alpha}(t;A+B)\ast f(t) \Big)\
	=\prescript{}{}{\mathcal{D}^{1}_{t}}\Big( C_{\alpha}(t;A+B)\ast f(t) \Big)\\& =\prescript{}{}{\mathcal{D}^{1}_{t}} C_{\alpha}(t;A+B)\ast f(t) +C_{\alpha}(0;A+B)f(t).
	\end{align*}
In a similar to \eqref{3.17}, we have
\begin{align}\label{end}
	\prescript{C}{}{\mathcal{D}^{\alpha}_{t}}C_{\alpha}(t;A+B)=(A+B)C_{\alpha}(t;A+B)=C_{\alpha}(t;A+B)(A+B).
\end{align}
	Since $C_{\alpha}(0;A+B)=I$ and using \eqref{end}, we attain 
	
	\begin{align*}
	\prescript{}{}{\mathcal{D}^{1}_{t}}C_{\alpha}(t;A+B)&=\prescript{}{}{\mathcal{I}^{1}_{t}}	\prescript{}{}{\mathcal{D}^{2}_{t}}C_{\alpha}(t;A+B)=	\prescript{}{}{\mathcal{I}^{\alpha-1}_{t}}\prescript{}{}{\mathcal{I}^{2-\alpha}_{t}}\prescript{}{}{\mathcal{D}^{2}_{t}}C_{\alpha}(t;A+B)
	=\prescript{}{}{\mathcal{I}^{\alpha-1}_{t}}\prescript{C}{}{\mathcal{D}^{\alpha}_{t}}C_{\alpha}(t;A+B)\\&=\prescript{}{}{\mathcal{I}^{\alpha-1}_{t}}\left( A+B\right) C_{\alpha}(t;A+B)=\prescript{}{}{\mathcal{I}^{\alpha-1}_{t}} C_{\alpha}(t;A+B)\left( A+B\right)\\&=(A+B)T_{\alpha}(t;A+B)=T_{\alpha}(t;A+B)(A+B).
	\end{align*}
	So, we derive a desired result:
	\begin{align*}
	\prescript{C}{}{\mathcal{D}^{\alpha}_{t}}w(t)&=\Big((A+B)T_{\alpha}(t;A+B)\Big)\ast f(t)+f(t)=\left( A+B\right)\int\limits_{0}^{t}T_{\alpha}(t-s;A+B)f(s)\mathrm{d}s+f(t)\\
	&=\Big( A+B\Big)w(t)+f(t).
	\end{align*}
	The proof is complete.
\end{proof}
\item In the particular case, if $B=0$, then the solution of the following abstract initial value problem for fractional evolution equation in a Banach space $\mathbb{X}$ which is fulfilling $u(0)=x$, $u^{\prime}(0)=y$ for $x, y \in \mathscr{D}(A)$:

\begin{equation}\label{eq-1}
\left( \prescript{C}{}{\mathcal{D}^{\alpha}_{t}}u\right) (t)=Au(t), \quad t>0,
\end{equation}
where a linear operator $A: \mathscr{D}(A)\subseteq \mathbb{X} \to \mathbb{X}$ is the infinitesimal generator of an operator fractional cosine function $C_{\alpha}$,

\begin{equation*}
u(t)=C_{\alpha}(t;A)x+S_{\alpha}(t;A)y, \quad t \in \mathbb{R}_{+},
\end{equation*}
where $S_{\alpha}:\mathbb{R}_{+}\to \mathcal{L}(\mathbb{X})$ denotes the fractional sine function associated with $C_{\alpha}$, defined by 

\begin{align*}
S_{\alpha}(t;A)x\coloneqq \int\limits_{0}^{t}C_{\alpha}(s;A)x\mathrm{d}s.
\end{align*}
It is important to note that this case have considered by Li in \cite{kexue1}. Furthermore, the same particular case has studied by Li et al. in \cite{kexue2} with the first initial condition is zero, i.e., $u(0)=0$ and by Chen and Li in \cite{chen} with the second initial condition is zero, i.e., $u^{\prime}(0)=0$.
\end{itemize}

\section{Special cases}\label{sec4}
In this section, we consider several special cases and remarks.

 \textbf{Case 1}. \textit{Classical case}: $\alpha=2$. 
 We consider the following abstract Cauchy problem for a second-order linear inhomogeneous evolution equation in a Banach space $\mathbb{X}$:
\begin{equation}\label{clas}
\begin{cases}
u^{\prime \prime}(t)= \left( A+B(t)\right) u(t) + f(t), \quad t\in \mathbb{R},\\
u(0)=x,\quad u^{\prime}(0)=y.
\end{cases}
\end{equation}
It should be note that in the case of $\alpha=2$, $T_{\alpha}(t;A)$ and $S_{\alpha}(t;A)$ coincide with the strongly continuous sine function $S(t;A)$. Furthermore, for $\alpha=2$, one and two parameter Mittag-Leffler type functions are converting to the hyperbolic cosine and sine functions, respectively:
\begin{align}
E_{2}(MK_{t}t^{2})=\sum_{k=0}^{\infty}\frac{M^{k}K^{k}_{t}t^{2k}}{(2k)!}=\cosh\Big(\sqrt{MK_{t}}t\Big), \quad t \in \mathbb{R},\\
tE_{2,2}(MK_{t}t^{2})=\sum_{k=0}^{\infty}\frac{M^{k}K^{k}_{t}t^{2k+1}}{(2k+1)!}=\frac{1}{\sqrt{MK_{t}}}\sinh\Big(\sqrt{MK_{t}}t\Big), \quad t \in \mathbb{R}.
\end{align}

\begin{theorem}\label{thm1-c}
	Let $A$ be the infinitesimal generator of the  an operator cosine function $\left\lbrace C(t;A): t \in \mathbb{R}\right\rbrace $ with \\ $\|C(t;A)\|\leq Me^{w|t|}$ for all $t \in \mathbb{R}$, and let $\left\lbrace S(t;A): t \in \mathbb{R}\right\rbrace $ denote the sine function associated with $C(t;A)$. Furthermore, suppose that $B:\mathbb{R}\to \mathcal{L}(\mathbb{X})$ and $f:\mathbb{R}\to \mathbb{X}$ are strongly continuously differentiable. Then
	
	\begin{itemize}
		\item The Cauchy problem for the $\mathbb{X}$-valued differential equation \eqref{clas}
		is uniformly well-posed; more precisely, for each pair $x,y\in \mathscr{D}(A)$ there is a uniquely determined solution $u \in \mathbb{C}^{2}\left(\mathbb{R};\mathscr{D}(A) \right)$ of \eqref{clas} fulfilling initial conditions and $u\in\mathscr{D}(A)$ depends continuously (with respect to the topology of uniform convergence) on compact subsets of $\mathbb{R}$ upon $(x,y)$. 
		\item $u\in\mathscr{D}(A)$ is given by
		\begin{equation}\label{usol-}
		u(t)= C(t;A+B)x+S(t;A+B)y+w(t), \quad t \in \mathbb{R},
		\end{equation}
		where for $x\in \mathscr{D}(A)$, $t \mapsto C(t;A+B)x$ and $t\mapsto S(t;A+B)x$ are $2$-times strongly continuously differentiable solutions of \eqref{clas} fulfilling 
		\begin{align}\label{csder-}
		&C(0;A+B)x=x, \quad C^{\prime}(0;A+B)x=0,\\
		&S(0;A+B)x=0, \quad S^{\prime}(0;A+B)x=x.
		\end{align}
		\item $C(\cdot;A+B),S(\cdot;A+B)\in \mathcal{L}(\mathbb{X})$  are defined by
		\begin{align}
		&C(t;A+B)x\coloneqq \sum_{n=0}^{\infty}C_{n}(t,A)x,\quad x\in\mathbb{X},\label{Cdefine-c}\\
		&S(t;A+B)x\coloneqq \sum_{n=0}^{\infty}S_{n}(t;A)x, \quad x\in \mathbb{X}\label{Sdefine-c},
		\end{align}
		where  
		\begin{align}
		&C_{0}(t;A)\coloneqq C(t;A),\quad S_{0}(t;A)\coloneqq S(t;A),\nonumber\\
		&C_{n}(t;A)x\coloneqq \int\limits_{0}^{t}S(t-s;A)B(s)C_{n-1}(s;A)x\mathrm{d}s\label{Calphan-c},\\
		&S_{n}(t;A)x\coloneqq \int\limits_{0}^{t}S(t-s;A)B(s)S_{n-1}(s;A)x\mathrm{d}s\label{Salphan-c}, \quad n\in \mathbb{N}, \quad x\in \mathbb{X}.
		\end{align}
		 \item A particular solution $w
		\in \mathscr{D}(A)$ of \eqref{clas} is defined by
	
		\begin{equation}
		w(t)=\sum_{n=0}^{\infty}w_{n}(t),
		\end{equation}
		where 
		\begin{align*}
		&w_{0}(t)=\int\limits_{0}^{t}S(t-s;A)f(s)\mathrm{d}s,\\
		&w_{n}(t)=\int\limits_{0}^{t}S(t-s;A)B(s)w_{n-1}(s)\mathrm{d}s, \quad n \in \mathbb{N}.
		\end{align*}
		
		\item With $K_{t}\coloneqq \sup\limits_{0\leq s \leq t}\Biggl\{ \norm{B(s)},\norm{B^{\prime}(s)}\Biggr\}$, we have for all $s\in [0,t]$, the bounds
		
		\begin{align*}
		&\norm{C(s;A+B)} \leq Me^{w|s|}\cosh\Big(\sqrt{MK_{t}}s\Big),\\
		&\norm{S(s;A+B)} \leq Me^{w|s|}\frac{1}{\sqrt{MK_{t}}}\sinh\Big(\sqrt{MK_{t}}|s|\Big),\\
		&\norm{C(s;A+B)-C(s;A)} \leq Me^{w|s|}\Big[\cosh\Big(\sqrt{MK_{t}}s\Big)-1\Big],\\
		&\norm{S(s;A+B)-S(s;A)} \leq Me^{w|s|}\Big[\frac{1}{\sqrt{MK_{t}}}\sinh\Big(\sqrt{MK_{t}}|s|\Big)-|s|\Big]. 
		\end{align*}
	\end{itemize}
\end{theorem}
In the special case, if $B(t)\coloneqq B \in \mathcal{L}(\mathbb{X})$ in $\mathbb{R}$, a particular solution of \eqref{clas} can be represented by as below:
\begin{align}
&w(t)=\int\limits_{0}^{t}S(t-s;A+B)f(s)\mathrm{d}s.
\end{align}

Note that the homogeneous case of \eqref{clas} have been considered by Lutz in \cite{lutz} and we have added a particular solution to \eqref{clas} as a consequence of the above results in fractional-order sense.

\textbf{Case 2.} If we consider $A$ and $B$ are bounded linear operators, then two subcases arise in this direction. In this case, for a bounded system operator $\mathscr{D}(A)$ coincide with the state space $\mathbb{X}$.

\textit{Case 2a: Non-permutable case} $\left( AB\neq BA\right) $.
It is known that fractional cosine, sine and Riemann-Liouville families can also be characterized via the Laplace transform of their generators.

It  should be note that the following lemma will be helpful to derive a solution to \eqref{ABeq}.

\begin{lemma}\label{non}
	For $A,B\in \mathcal{L}(\mathbb{X})$ satisfying $AB\neq BA$, we have:
	\begin{align*}
	\mathscr{L}^{-1}\Biggl\{ \lambda^{\gamma}\left[ (\lambda^{\alpha}I-A)^{-1}B\right]^{k} (\lambda^{\alpha}I-A)^{-1}\Biggr\}(t)=\sum_{m=0}^{\infty}\frac{\mathcal{Q}^{A,B}_{k,m}}{\Gamma(k\alpha+m\alpha+\alpha-\gamma)}t^{k\alpha+m\alpha+\alpha-\gamma-1}, \quad k\in \mathbb{N}_{0},
	\end{align*}
 where 
$\mathcal{Q}^{A,B}_{k,m}\in \mathcal{L}(\mathbb{X})$ is defined by 
\begin{align}
	\mathcal{Q}^{A,B}_{k,m}=\sum_{l=0}^{k}A^{k-l}
	B\mathcal{Q}^{A,B}_{l,m-1}, k,m \in \mathbb{N}, \quad \mathcal{Q}^{A,B}_{k,0}=A^{k}, k \in \mathbb{N}_{0}, \quad \mathcal{Q}^{A,B}_{0,m}=B^{m}, m\in \mathbb{N}_{0}.
\end{align}
\end{lemma}

\begin{proof}
This lemma is a special case of Lemma 3.2 in  \cite{arxiv,authorea} with $\beta=0$ which can be proved via mathematical induction principle. So, the proof of this lemma is omitted here.	
\end{proof}

 Then, we apply Laplace transform method and solve the equation \eqref{ABeq} with bounded linear operators $A,B$. After solving the equation with respect to $U(\lambda)$, we obtain

\begin{align}\label{ulap}
U(\lambda)=\lambda^{\alpha-1}(\lambda^{\alpha}I-A-B)^{-1}x+\lambda^{\alpha-2}(\lambda^{\alpha}I-A-B)^{-1}y+(\lambda^{\alpha}I-A-B)^{-1}F(\lambda),
\end{align}
where $U(\lambda)=\left( \mathscr{L} u\right) (\lambda)$ and $F(\lambda)=\left( \mathscr{L} f\right) (\lambda)$.

Then for non-permutable linear operators $A,B\in \mathcal{L}(\mathbb{X})$ and sufficiently large $\lambda$, we have

\begin{align*}
(\lambda^{\alpha}I-A-B)^{-1}&=\Big[ (\lambda^{\alpha}I-A)(I-(\lambda^{\alpha}I-A)^{-1}B)\Big]^{-1}\\
&=(I-(\lambda^{\alpha}I-A)^{-1}B)^{-1}(\lambda^{\alpha}I-A)^{-1}\\
&=\sum_{k=0}^{\infty}\left[ (\lambda^{\alpha}I-A)^{-1}B\right]^{k} (\lambda^{\alpha}I-A)^{-1},
\end{align*}
where $\|(\lambda^{\alpha}I-A)^{-1}B\|<I$.

Taking inverse Laplace transform of \eqref{ulap}, we have 

\begin{align*}
u(t)&=\mathscr{L}^{-1}\Biggl\{ \sum_{k=0}^{\infty}\lambda^{\alpha-1}\left[ (\lambda^{\alpha}I-A)^{-1}B\right]^{k} (\lambda^{\alpha}I-A)^{-1}\Biggr\}(t)x\\
&+\mathscr{L}^{-1}\Biggl\{\sum_{k=0}^{\infty}\lambda^{\alpha-2}\left[ (\lambda^{\alpha}I-A)^{-1}B\right]^{k} (\lambda^{\alpha}I-A)^{-1}\Biggr\}(t)y\\
&+\mathscr{L}^{-1}\Biggl\{\sum_{k=0}^{\infty}\left[ (\lambda^{\alpha}I-A)^{-1}B\right]^{k} (\lambda^{\alpha}I-A)^{-1}F(\lambda)\Biggr\}(t)
\end{align*}
According to the Lemma \ref{non}, we derive the explicit analytical representation of the solution $u(t)$ to \eqref{ABeq} with  linear bounded operators $A$ and $B$ as follows:

\begin{align}\label{uu}
u(t)&=\sum_{n=0}^{\infty}\sum_{\substack{k+m=n\\ k,m \geq 0}}\mathcal{Q}^{A,B}_{k,m}\frac{t^{n\alpha}}{\Gamma(n\alpha+1)}x\nonumber\\
&+\sum_{n=0}^{\infty}\sum_{\substack{k+m=n\\ k,m \geq 0}}\mathcal{Q}^{A,B}_{k,m}\frac{t^{n\alpha+1}}{\Gamma(n\alpha+2)}y\\
&+\int\limits_{0}^{t}\sum_{n=0}^{\infty}\sum_{\substack{k+m=n\\ k,m \geq 0}}\mathcal{Q}^{A,B}_{k,m}\frac{(t-s)^{n\alpha+\alpha-1}}{\Gamma(n\alpha+\alpha)}f(s)\mathrm{d}s. \nonumber
\end{align}
Thus, we proved the following theorem:
\begin{theorem}
	Let $A,B\in \mathcal{L}(\mathbb{X})$ with non-zero commutator, i.e., $[A, B]= AB-BA\neq 0$. Then the closed-form of a solution $u\in \mathbb{X}$ of the abstract Cauchy problem \eqref{ABeq} can be expressed as 
	\begin{align*}
		u(t)=\sum_{n=0}^{\infty}\sum_{\substack{k+m=n\\ k,m \geq 0}}\mathcal{Q}^{A,B}_{k,m}\frac{t^{n\alpha}}{\Gamma(n\alpha+1)}x&+\sum_{n=0}^{\infty}\sum_{\substack{k+m=n\\ k,m \geq 0}}\mathcal{Q}^{A,B}_{k,m}\frac{t^{n\alpha+1}}{\Gamma(n\alpha+2)}y\\
		&+\int\limits_{0}^{t}\sum_{n=0}^{\infty}\sum_{\substack{k+m=n\\ k,m \geq 0}}\mathcal{Q}^{A,B}_{k,m}\frac{(t-s)^{n\alpha+\alpha-1}}{\Gamma(n\alpha+\alpha)}f(s)\mathrm{d}s. 
	\end{align*}
\end{theorem}
\textit{Case 2b: Permutable case} $\left( AB= BA\right) $. This case is obtained directly from \eqref{uu} using the following identity for $\mathcal{Q}^{A,B}_{k,m}$ \cite{arxiv,authorea}:
\begin{equation*}
\mathcal{Q}^{A,B}_{k,m}=\binom{k+m}{m}A^{k}B^{m},
\end{equation*}
 and the eminent Cauchy product formula, we acquire 
\allowdisplaybreaks
\begin{align*}
u(t)&=\sum_{k=0}^{\infty}\sum_{m=0}^{\infty}\binom{k+m}{m}A^{k}B^{m}\frac{t^{(k+m)\alpha}}{\Gamma((k+m)\alpha+1)}x\nonumber\\
&+\sum_{k=0}^{\infty}\sum_{m=0}^{\infty}\binom{k+m}{m}A^{k}B^{m}\frac{t^{(k+m)\alpha+1}}{\Gamma((k+m)\alpha+2)}y\\
&+\int\limits_{0}^{t}\sum_{k=0}^{\infty}\sum_{m=0}^{\infty}\binom{k+m}{m}A^{k}B^{m}\frac{(t-s)^{(k+m)\alpha+\alpha-1}}{\Gamma((k+m)\alpha+\alpha)}f(s)\mathrm{d}s\nonumber\\
&=\sum_{k=0}^{\infty}\sum_{m=0}^{k}\binom{k}{m}A^{k-m}B^{m}\frac{t^{k\alpha}}{\Gamma(k\alpha+1)}x\nonumber\\
&+\sum_{k=0}^{\infty}\sum_{m=0}^{k}\binom{k}{m}A^{k-m}B^{m}\frac{t^{k\alpha+1}}{\Gamma(k\alpha+2)}y\\
&+\int\limits_{0}^{t}\sum_{k=0}^{\infty}\sum_{m=0}^{k}\binom{k}{m}A^{k-m}B^{m}\frac{(t-s)^{k\alpha+\alpha-1}}{\Gamma(k\alpha+\alpha)}f(s)\mathrm{d}s.\nonumber
\end{align*}
Now, we apply the well-known binomial theorem for permutable linear bounded operators $A,B\in \mathcal{L}(\mathbb{X})$:
\begin{equation*}
\sum_{m=0}^{k}\binom{k}{m}A^{k-m}B^{m}=(A+B)^{k}.
\end{equation*}
Then we acquire an explicit analytical representation of the solution to \eqref{ABeq} with the help of Mittag-Leffler type functions generated by linear bounded operators:

\begin{align*}
u(t)=E_{\alpha,1}((A+B)t^{\alpha})x+tE_{\alpha,2}((A+B)t^{\alpha})y+\int\limits_{0}^{t}(t-s)^{\alpha-1}E_{\alpha,\alpha}((A+B)(t-s)^{\alpha})f(s)\mathrm{d}s.
\end{align*}
Therefore, we proved the following theorem.
\begin{theorem}\label{thm2}
	Let $A,B \in \mathcal{L}(\mathbb{X})$ with zero commutator, i.e., $[A, B]=AB-BA= 0$. A classical solution $u\in \mathbb{X}$ of the abstract initial value problem \eqref{ABeq} can be expressed as
	
	\begin{align*}
	u(t)=E_{\alpha,1}((A+B)t^{\alpha})x+tE_{\alpha,2}((A+B)t^{\alpha})y+t^{\alpha-1}E_{\alpha,\alpha}((A+B)t^{\alpha})\ast f(t).
	\end{align*}
\end{theorem}

\begin{remark}
	 A classical case with non-permutable and permutable linear bounded operators can be derived as a result of above Case 2a and 2b. Therefore, we consider the following abstract Cauchy problem:
	 \begin{equation}\label{class}
	 	\begin{cases}
	 		u^{\prime \prime}(t)= \left( A+B\right) u(t) + f(t), \quad t\in \mathbb{R},\\
	 		u(0)=x,\quad u^{\prime}(0)=y.
	 	\end{cases}
	 \end{equation}
\end{remark}
\begin{theorem}
	Let $A,B\in \mathcal{L}(\mathbb{X})$ with non-zero commutator, i.e., $[A, B]= AB-BA\neq 0$. Then the closed-form of a solution $u\in \mathbb{X}$ of the abstract Cauchy problem \eqref{class} can be expressed as 
	\begin{align*}
		u(t)=\sum_{n=0}^{\infty}\sum_{\substack{k+m=n\\ k,m \geq 0}}\mathcal{Q}^{A,B}_{k,m}\frac{t^{2n}}{(2n)!}x&+\sum_{n=0}^{\infty}\sum_{\substack{k+m=n\\ k,m \geq 0}}\mathcal{Q}^{A,B}_{k,m}\frac{t^{2n+1}}{(2n+1)!}y\\
		&+\int\limits_{0}^{t}\sum_{n=0}^{\infty}\sum_{\substack{k+m=n\\ k,m \geq 0}}\mathcal{Q}^{A,B}_{k,m}\frac{(t-s)^{2n+1}}{(2n+1)!}f(s)\mathrm{d}s. 
	\end{align*}
\end{theorem}
\begin{theorem}
	Let $A,B \in \mathcal{L}(\mathbb{X})$ with zero commutator, i.e., $[A, B]=AB-BA= 0$. A classical solution $u\in \mathbb{X}$ of the abstract initial value problem \eqref{class} can be expressed as
	
	\begin{align*}
		u(t)=\cos(\sqrt{A+B}t)x+\frac{1}{\sqrt{A+B}}\sin(\sqrt{A+B}t)y+\frac{1}{\sqrt{A+B}}\sin(\sqrt{A+B}t)\ast f(t).
	\end{align*}
\end{theorem}


\begin{thebibliography}{99}
	
\bibitem{Fit} 
 W.E. Fitzgibbon,
Global existence and boundedness of solutions to the extensible beam equation, 
SIAM J. Math. Anal., 13 (1982) 739-745.

\bibitem{W-K}
S. Woinowsky-Krieger,
The effect of an axial force on the vibration of hinged bars, ASME J. Appl. Mech., 17 (1950) 35-36.

\bibitem{MM}
N.I. Mahmudov, M.A. McKibben, 
Abstract second-order damped McKean-Vlasov stochastic evolution equations, 
Stoc. Anal. Appl., 242 (2006) 303-328.

\bibitem{Fat-1}
H.O. Fattorini, 
Ordinary differential equations in linear topological spaces, I, J. Differential Equations, 5 (1968) 72-105.

\bibitem{Fat-2}
H.O. Fattorini, 
Ordinary differential equations in linear topological spaces, II, J. Differential Equations, 6 (1969) 50-70.

	
	\bibitem{TW1}
C.C Travis, G.F. Webb,
Cosine families and abstract nonlinear second order differential equations, 
Acta Math. Acad. Scient. Hung., 32 (3-4) (1978) 75-96.
	
\bibitem{TW2}
C.C. Travis, G.F. Webb, 
Compactness, regularity and uniform continuity properties of strongly continuous cosine families, Houston J.  Math., 3 (1977) 555-567.
	
\bibitem{bochenek}
J. Bochenek, An abstract nonlinear second order differential equation, Annales Polonici Mathematici,  54 (1991) 155-166.
	
	
	
\bibitem{kexue1}
K. Li, Fractional order semilinear Volterra integrodifferential equations in Banach spaces, Topol. Methods Nonlinear Anal., 47 (2) (2016) 439-455.


\bibitem{kexue2}
K. Li, J. Peng, J. Jia, Cauchy problems for fractional differential equations with Riemann-Liouville derivatives, J. Funct. Anal., 263 (2012) 476-510. 	

\bibitem{chen}
C. Chen, M. Li, On fractional resolvent operator functions, Semigroup forum, 80 (2010) 121-142.

\bibitem{henriquez}
H.R. Henr\'{i}quez, J.G. Mesquita, J.C. Poza, Existence of solutions of the abstract Cauchy problem of fractional order, J. Funct. Anal., 281 (2021) 109028. 
	
\bibitem{philips}
R.S. Phillips, Perturbation theory for semi-groups of linear operators. Trans. Am. Math. Soc., 74 (1954) 199-221.

\bibitem{lutz}
D. Lutz, On bounded time-dependent perturbations of operator cosine functions, Aequationes Mathematicae, 23 (1981) 197-203.

\bibitem{TW3}
C.C. Travis, G.F. Webb, Perturbation of strongly continuous cosine family generators, Colloquium Mathematicae, 45 (2) (1981) 277-285.	

\bibitem{lin}
Y. Lin, Time-dependent perturbation theory for abstract evolution equations of second order, Studio Mathematica, 130 (1998), 263-274.

\bibitem{watanabe2}
H. Serizawa, M. Watanabe, Time-dependent perturbation for cosine families in Banach spaces, Houstan Journal of Mathematics, 12 (4) (1986) 579-586.
	
	
\bibitem{bazhlekova}
E. Bazhlekova, Fractional evolution equations in Banach spaces, Ph.D. Thesis, Eindhoven University of Technology, 2001.
	
	
\bibitem{bazhlekova2}
E. Bazhlekova, Perturbation properties for abstract evolution equations of fractional order, Fract. Cal. Appl. Anal., 2 (4) (1999) 359-366.
	
	
	
	

\bibitem{kilbas}
A.A. Kilbas, H.M. Srivastava, J.J. Trujillo, Theory and applications of fractional differential equations, Elsevier Sceince B.V., 2006. 

\bibitem{podlubny}
I. Podlubny, Fractional differential equations, Academic Press, New York, 1999.



\bibitem{engel}
K.-J. Engel, R. Nagel, One-Parameter semigroups for linear evolution equations, vol. 194, Graduate Texts in Mathematics, Springer-Verlag, New York, 2000.

\bibitem{kato}
T. Kato, Perturbation theory for linear operators, Berlin, Heidelberg, New York, Springer, 1966.


\bibitem{watanabe}
M. Watanabe, A Perturbation theory for abstract evolution equations of second order, vol. 58, Proc. Japan Acad., 1982.


\bibitem{voigt}
J. Voigt, On the perturbation theory for strongly continuous semigroups, Math. Ann., (229) (1977) 163-171.

\bibitem{miyadera}
M. Shimizu, I. Miyadera, Perturbation Theory for Cosine Families on Banach Spaces, Tokyo J. Math. 1 (2) (1978) 333-343.

\bibitem{hille-philips}
E. Hille, R.S. Philips, Functional analysis and semigroups, vol. 31, Revised Ed. Providence: Am. Math. Soc. Colloq. Publ.,  1957.


\bibitem{Gorenflo}
R. Gorenflo, A.A. Kilbas, F. Mainardi, S.V. Rogosin, Mittag-Leffler functions, related topics and applications, Springer-Verlag, Berlin, 2014.



\bibitem{arxiv}
N.I. Mahmudov, A. Ahmadova, I.T Huseynov, A new technique for solving Sobolev type fractional
multi-order evolution equations, arXiv preprint arXiv:2102.10318.

\bibitem{authorea}
I.T. Huseynov, A. Ahmadova, N.I. Mahmudov, On a study of Sobolev type fractional functional evolution equations, Authorea, M2021,
DOI: 10.22541/au.161562420.01059626/v1.

\end{thebibliography}
\end{document}